\documentclass[11pt,a4paper]{article}

\usepackage[utf8]{inputenc}
\usepackage[T1]{fontenc}
\usepackage[english]{babel}

\usepackage{amsmath,amssymb,amsthm}

\usepackage{geometry}
\usepackage{tikz}
\usetikzlibrary{arrows.meta}
\usepackage{float}
\usepackage{pgfplots}
\pgfplotsset{compat=1.16}
\usepackage{array}
\usepackage{booktabs}

\usepackage{enumitem}

\usepackage{graphicx}

\usepackage{hyperref}
\hypersetup{
  colorlinks=true,
  linkcolor=blue!50!black,
  citecolor=blue!50!black,
  urlcolor=blue!50!black,
  pdfencoding=auto,
  hypertexnames=false,
  breaklinks=true
}

\newtheorem{theorem}{Theorem}[section]
\newtheorem{proposition}[theorem]{Proposition}
\newtheorem{lemma}[theorem]{Lemma}
\newtheorem{corollary}[theorem]{Corollary}

\newtheorem{openproblem}{Open Problem}[section]

\theoremstyle{definition}
\newtheorem{definition}[theorem]{Definition}

\theoremstyle{remark}
\newtheorem{remark}[theorem]{Remark}

\setlist[enumerate]{leftmargin=*}
\setlist[itemize]{leftmargin=*}

\title{
	Non-Autonomous Return Dynamics for Nested Convex Bodies:\\
	Cubic Tangency, Summable Geometric Variation, and Critical Regimes
}

\author{
	M. El Morsalani\thanks{QWave Consult, Germany. \texttt{Mohamed.elmorsalani@qwave-consult.eu}} \
	\\
     M. Barkatou\thanks{ISTM Laboratory, Chouaib Doukkali University, Morocco. \texttt{barkatou.m@ucd.ac.ma}}   
    }

\date{October 2026}

\begin{document}
\maketitle

\begin{abstract}
	We study a non-autonomous geometric dynamics generated by successive normal
transitions between nested convex boundaries sharing a common tangency point.
	If
	\[
	C_{k+1}\subset C_k\subset\mathbb{R}^2
	\]
	and \(\mathcal G_k:\partial C_k\to\partial C_{k+1}\) denotes the first intersection
	of the inward normal ray with the next boundary, then the induced scalar map
	in signed arclength coordinates is tangent to the identity to second order.
	Writing the boundaries locally as graphs \(f_k\), we prove the uniform defect
	estimate
	\[
	|G_k(s)-s|
	\leq C\|f_{k+1}-f_k\|_{C^2}|s|^3.
	\]
	The first nonlinear term is cubic,
	\[
	G_k(s)=s+\beta_k s^3+O(s^4),
	\qquad
	\beta_k=
	\frac{(\kappa_{k+1}-\kappa_k)(\kappa_{k+1}-2\kappa_k)}{6},
	\]
	and higher coefficients satisfy a diagonal-vanishing principle determined
	by the variation of the corresponding boundary jets.
	
	When the geometric variation is summable,
	\[
	\sum_{k\geq0}\|f_{k+1}-f_k\|_{C^2}<\infty,
	\]
	the complete transition defect is summable. Every sufficiently small
	nonzero orbit has finite total drift and converges to a nonzero limit;
	hence the common tangency point is Lyapunov stable but not locally
	attractive. The associated non-autonomous composition therefore defines a
	limiting displacement map tangent to the identity.
	
	We also study complementary abstract parabolic recurrences with
	nonsummable nonlinear drift. A weighted comparison principle yields the
	corresponding polynomial and logarithmic asymptotic laws, including
	critical quartic and quintic regimes. These abstract regimes are kept
	separate from the geometric theory, leading to the further problem of
	determining which critical parabolic behaviors can actually be realized by
	nested convex boundaries.
\end{abstract}

\medskip

\tableofcontents

% ============================================================
\section{Introduction}
\label{sec:introduction}

Geometrically generated discrete dynamical systems arise naturally when
successive states are related by projection, normal transport, or other
boundary-based constructions.  In convex geometry, the interaction between
normal fields, curvature, and nearby boundaries provides a particularly
rigid mechanism for producing such maps; see, for example,
\cite{Gruber2007,HugWeil2020,Schneider2014}.  When the underlying geometry is
fixed, iteration leads to an autonomous dynamical system.  A different
problem appears when the geometry itself changes from one step to the next.
One is then led to a sequence of maps
\[
\mathcal G_k:\partial C_k\longrightarrow \partial C_{k+1},
\]
and the relevant dynamical object is the non-autonomous composition
\[
\mathcal G_{n-1}\circ\cdots\circ \mathcal G_1\circ \mathcal G_0.
\]
The purpose of the present paper is to study this situation for nested convex
boundaries meeting at a common tangency point and to determine how the
variation of the geometry controls the accumulated dynamics.

The problem is a natural continuation of the return-map program developed
in \cite{Barkatou2002,BarkatouElMorsalani2026,
	ElMorsalaniBarkatou2026,ElMorsalaniBarkatouMorse2026,
	ElMorsalaniBarkatouInverse2026}.  For a fixed convex core $C$ and an
admissible outer domain in the class $\mathcal{O}_C$, the return map
introduced in \cite{BarkatouElMorsalani2026} is generated by a geometric
round trip between the two boundaries.  Its local structure is governed by
the thickness function and by the curvature-dependent operator associated
with the convex core.  The subsequent global analysis
\cite{ElMorsalaniBarkatou2026} identifies regimes in which the thickness is
a strict Lyapunov quantity, establishes convergence of return orbits, and
describes the loss of attraction when the corresponding spectral threshold
is crossed.  The relation between this dynamics and the critical-point
structure of the thickness landscape is developed further in
\cite{ElMorsalaniBarkatouMorse2026}, where Morse indices, stable and
unstable directions, fixed-point indices, and global topological constraints
are linked to the exact return map.  Finally,
\cite{ElMorsalaniBarkatouInverse2026} considers the reverse question:
which geometric information can be reconstructed from the observed return
dynamics?

All four of these problems concern a fixed geometric configuration.  The
present work introduces a different level of variation.  We consider
\[
C_{k+1}\subset C_k\subset\mathbb{R}^2,
\qquad k\geq0,
\]
whose boundaries share a common point $p$ and a common tangent line there.
Instead of iterating a single return map on one fixed boundary, we follow
the normal transition from $\partial C_k$ to $\partial C_{k+1}$.  Thus both
the domain of the map and the map itself change with $k$.  This converts the
autonomous geometric dynamics of the preceding papers into a genuinely
non-autonomous problem and raises two related questions: how much displacement
can infinitely many weak geometric transitions accumulate, and under what
conditions can such transitions drive an orbit toward the common tangency
point?

Non-autonomous discrete systems form a well-developed subject in their own
right; general frameworks may be found in
\cite{Kloeden2011,Potzsche2010}.  More recent work has emphasized the
dynamics of compositions of varying maps, including random and
non-autonomous iteration on complex domains
\cite{AbateChristodoulou2022,BeniniEtAl2024}.  A second relevant body of
work concerns maps that are tangent, or close, to the identity.  Such maps
occur naturally in averaging and interpolation problems
\cite{GelfreichVieiro2018}, while the asymptotics of real difference
equations near parabolic equilibria are studied in
\cite{CollGasullProhens2020}.  Parabolic and tangent-to-the-identity
phenomena also remain an active topic in higher-dimensional complex
dynamics; see, for example, \cite{AstorgBocThaler2026}.  These settings are
mathematically different from the one considered here, but they highlight
two structural features that are central to the present problem: the
composition of maps depending on the iteration index and the dominant role
of the first nonlinear term when the linearized dynamics is neutral.

The geometric origin of the maps considered here imposes substantially more
structure than an arbitrary sequence of near-identity maps.  The geometric
map $\mathcal G_k$ is not prescribed independently; it is the first
intersection of the inward normal ray to $\partial C_k$ with the next
boundary $\partial C_{k+1}$.  Moreover, the departure of the induced scalar
map $G_k$ from the identity is generated by the difference between two
consecutive boundaries.  One of the main points of
the paper is that these two facts can be combined quantitatively.  If the
boundaries are written locally as graphs $f_k$ and
\[
\varepsilon_k
:=
\|f_{k+1}-f_k\|_{C^2},
\]
then, in signed arclength coordinates near the common tangency point,
\[
|G_k(s)-s|
\leq
C\varepsilon_k |s|^3.
\]
Thus tangency eliminates both the constant and quadratic defects and forces
the complete transition map to be parabolic:
\[
G_k(0)=0,\qquad
G_k'(0)=1,\qquad
G_k''(0)=0.
\]
The estimate controls the full transition defect rather than merely a finite
Taylor truncation.

A more precise expansion shows how the geometry enters the first nonlinear
terms.  Writing
\[
\kappa_k=-f_k''(0),
\]
the cubic coefficient is
\[
G_k(s)
=
s+\beta_k s^3+O(s^4),
\qquad
\beta_k
=
\frac{(\kappa_{k+1}-\kappa_k)
	(\kappa_{k+1}-2\kappa_k)}{6}.
\]
Moreover, nesting and the uniform curvature bound imply
\[
\sum_{k=0}^{\infty}|\beta_k|<\infty,
\]
even without summability of the full \(C^2\)-variation.

Hence the first nonlinear transition is determined entirely by the two
curvatures at the common tangency point.  Without reflection symmetry, the
next contribution is quartic and detects the third boundary jets; under
local symmetry all even coefficients vanish and the next possible
contribution is quintic.  More generally, every nonidentity Taylor
coefficient satisfies a diagonal-vanishing principle: it disappears when
the corresponding jets of two consecutive boundaries agree.  The
near-identity character of the dynamics is therefore not imposed
analytically but produced by convergence of the geometry itself.

This distinction becomes decisive for the long-time behavior.  If the
geometric variation is summable,
\[
\sum_{k=0}^{\infty}
\|f_{k+1}-f_k\|_{C^2}
<\infty,
\]
then the complete transition defects are summable as well.  We prove that
every sufficiently small nonzero orbit has finite total drift and converges
to a nonzero limit,
\[
s_k\longrightarrow s_\infty\neq0.
\]
Consequently, the common tangency point is Lyapunov stable but not locally
attractive.  The infinite non-autonomous composition nevertheless has a
well-defined limiting effect,
\[
\mathcal{L}(s)
:=
\lim_{n\to\infty}
G_{n-1}\circ\cdots\circ G_0(s),
\qquad
\mathcal L(s)
=
s+
\left(\sum_{k=0}^{\infty}\beta_k\right)s^3
+o(s^3).
\]
Thus summably changing geometry produces a finite nonlinear displacement,
rather than asymptotic collapse onto the common tangency point.

The complementary regime is obtained when the effective parabolic drift is
not summable.  To separate the dynamical mechanism from the question of
geometric realizability, we study abstract non-autonomous recurrences of the
form
\[
s_{k+1}
=
s_k-a_k s_k^{m+1}+R_k(s_k).
\]
Under remainder estimates matched to the scale of the leading coefficient,
the transformed variable $s_k^{-m}$ is asymptotic to
$m\sum_{j<k}a_j$.  This gives the usual polynomial parabolic laws when the
partial sums grow algebraically and produces logarithmic laws at the
critical harmonic scale.  In particular, the quartic and quintic critical
models lead respectively to
\[
s_k\sim(3c\log k)^{-1/3},
\qquad
s_k\sim(4c\log k)^{-1/4}.
\]
These abstract regimes are deliberately distinguished from the summable
geometric regime.  Their realization by an actual nested convex family
requires not only the appropriate Taylor coefficient but also geometric
control of the complete higher-order transition defect.

There are several reasons for considering this non-autonomous extension.
First, sequences of nearby boundaries occur naturally when a shape is
successively approximated, deformed, or optimized.  Normal transport,
curvature, and changing interfaces are basic ingredients in modern shape
optimization and evolving-surface analysis
\cite{PlotnikovSokolowski2023,BaiHuLi2024}.  The results obtained here are
not presented as numerical schemes for those problems; rather, these areas
provide a broader setting in which understanding the cumulative effect of
successive normal transitions is geometrically relevant.  Second, the
present construction provides a dynamic extension of the earlier
$\mathcal{O}_C$ program: local return geometry is replaced by a sequence of
geometric transition maps, autonomous Lyapunov behavior by accumulated
non-autonomous drift, and the fixed thickness landscape by a changing
sequence of boundary jets.  Third, the limiting map $\mathcal{L}$ naturally
reopens the inverse question of
\cite{ElMorsalaniBarkatouInverse2026}: to what extent does the accumulated
transition encode the sequence of geometric changes that generated it?
This suggests an inverse problem for varying geometry rather than for one
fixed domain.

The contribution of the paper is therefore not a new general theory of
non-autonomous or parabolic maps.  Its specific point is the combination of
three structures: normal transition between nested convex boundaries,
parabolic tangency forced by the geometry, and quantitative control of the
non-autonomous defect by the variation of consecutive boundaries.  Within
this framework, summability of the geometry separates finite accumulated
drift from the critical nonsummable mechanisms that can produce attraction
to the neutral state.

The paper is organized as follows.  Section~2 introduces the geometric
transition between consecutive nested boundaries, proves offset
independence, and derives the uniform transition-defect estimate.
Section~3 obtains the local parabolic normal form, including the cubic
tangency law, the nonsymmetric quartic contribution, and the
diagonal-vanishing principle.  Section~4 studies summable geometric
variation and proves convergence to a nonzero limiting state together with
Lyapunov stability of the common tangency point.  Section~5 develops the
abstract parabolic comparison principle and the cubic, quartic, and quintic
critical regimes.  Section~6 provides an explicit local parabolic boundary-germ model,
exact normal-intersection computations, and numerical checks separating
the geometric and synthetic regimes.  Section~7 discusses geometric
realizability, borderline variation, nonsymmetric and higher-dimensional
extensions, possible nontrivial limit sets, and the asymptotic displacement
map.  Section~8 concludes the paper.
% ============================================================

% ================================================================
\section{Geometric Transition Maps}
\label{sec:geometric-transition}
% ================================================================

We first isolate the geometric structure of the transition from one
convex boundary to the next.  The analysis in this paper is local
near a common tangency point and is carried out in dimension two.
The higher-dimensional problem will be discussed separately.

Let
\[
C_{k+1}\subset C_k\subset\mathbb R^2,
\qquad k\ge0,
\]
be convex bodies with sufficiently smooth boundaries.  Assume that
all boundaries contain a common point
\[
p\in\bigcap_{k\ge0}\partial C_k
\]
and have the same tangent line at \(p\).

Choose fixed orthonormal coordinates in which
\[
p=(0,0),
\]
the common tangent is the \(x\)-axis, and the common outward normal
points in the positive \(y\)-direction.  On a fixed interval
\[
I_{\rho_0}:=(-\rho_0,\rho_0),
\]
write
\begin{equation}
\partial C_k:
\qquad
c_k(x)=\bigl(x,f_k(x)\bigr),
\label{eq:local-graph-Ck}
\end{equation}
with
\begin{equation}
f_k(0)=0,
\qquad
f_k'(0)=0.
\label{eq:common-tangency-data}
\end{equation}
Locally,
\[
C_k=\{(x,y):y\le f_k(x)\}.
\]
Hence nesting implies
\begin{equation}
f_{k+1}(x)\le f_k(x)
\label{eq:local-nesting-functions}
\end{equation}
whenever both graphs are defined.

We assume throughout this section a uniform bound
\begin{equation}
\sup_{k\ge0}\|f_k\|_{C^2(I_{\rho_0})}
\le M_2.
\label{eq:uniform-C2-bound}
\end{equation}
Higher regularity will be imposed later when Taylor coefficients of
the transition map are computed.

% ---------------------------------------------------------------
\subsection{Nested boundaries and the normal transition}
\label{subsec:normal-transition}
% ---------------------------------------------------------------

The outward unit normal to \(\partial C_k\) at \(c_k(x)\) is
\begin{equation}
\nu_k(x)
=
\frac{\bigl(-f_k'(x),1\bigr)}
{\sqrt{1+f_k'(x)^2}}.
\label{eq:normal-Ck}
\end{equation}

For \(c\in\partial C_k\) near \(p\), consider the inward normal ray
\begin{equation}
\mathcal R_k(c)
:=
\{c-t\nu_k(c):t\ge0\}.
\label{eq:inward-normal-ray}
\end{equation}
The local intersection with the next boundary does not need to be imposed
as an additional hypothesis.  It follows uniformly from nesting, common
tangency, and the $C^2$ bound.

\begin{lemma}[Uniform local existence and uniqueness of the normal transition]
	\label{lem:uniform-normal-intersection}
	There exists $r_0>0$, depending only on the uniform bound~\eqref{eq:uniform-C2-bound}
	and on the common graph neighbourhood, such that for every $k$ and every
	$|x|\leq r_0$, the inward normal ray issued from
	\[
	c_k(x)=(x,f_k(x))
	\]
	meets $\partial C_{k+1}$ at a unique first point contained in the common
	local chart.
	
	More precisely, there exists a unique $\lambda_k(x)\geq0$ such that
	\[
	\bigl(
	x+\lambda_k(x)f_k'(x),
	f_k(x)-\lambda_k(x)
	\bigr)
	\in\partial C_{k+1},
	\]
	and
	\[
	0\leq\lambda_k(x)
	\leq
	\frac{2}{3}\,
	\|f_{k+1}-f_k\|_{C^2(I_{\rho_0})}|x|^2 .
	\]
	Consequently,
	\[
	\left|
	x+\lambda_k(x)f_k'(x)-x
	\right|
	\leq
	C\|f_{k+1}-f_k\|_{C^2(I_{\rho_0})}|x|^3 ,
	\]
	with a constant $C$ independent of $k$.
\end{lemma}

\begin{proof}
	Set
	\[
	h_k:=f_k-f_{k+1}.
	\]
	By nesting and common tangency,
	\[
	h_k\geq0,
	\qquad
	h_k(0)=h_k'(0)=0.
	\]
	Writing
	\[
	\varepsilon_k
	:=
	\|f_{k+1}-f_k\|_{C^2(I_{\rho_0})},
	\]
	Taylor's theorem gives
	\[
	0\leq h_k(x)\leq \frac12\varepsilon_k |x|^2.
	\]
	
	For fixed $k$ and $x$, define
	\[
	F_{k,x}(\lambda)
	:=
	f_k(x)-\lambda
	-
	f_{k+1}\bigl(x+\lambda f_k'(x)\bigr).
	\]
	A zero of $F_{k,x}$ is exactly an intersection of the inward normal line
	with $\partial C_{k+1}$.
	
	Since
	\[
	F_{k,x}(0)=h_k(x)\geq0,
	\]
	it remains to show that $F_{k,x}$ decreases uniformly.  By the uniform
	$C^2$ bound and the conditions $f_j'(0)=0$, we may choose $r_0>0$ so
	small that
	\[
	|f_j'(y)|\leq\frac12
	\]
	for every $j$ and every $y$ in the relevant common neighbourhood.
	After reducing $r_0$ once more, the points
	\[
	x+\lambda f_k'(x),
	\qquad
	0\leq\lambda\leq\frac43h_k(x),
	\]
	remain in that neighbourhood uniformly in $k$.  Hence
	\[
	\frac{\partial F_{k,x}}{\partial\lambda}
	=
	-1
	-
	f_{k+1}'\bigl(x+\lambda f_k'(x)\bigr)f_k'(x)
	\leq
	-\frac34 .
	\]
	Thus $F_{k,x}$ is strictly decreasing.  With
	\[
	\lambda_*:=\frac43h_k(x),
	\]
	we have
	\[
	F_{k,x}(\lambda_*)
	\leq
	F_{k,x}(0)-\frac34\lambda_*
	=
	0.
	\]
	The intermediate-value theorem therefore gives a unique
	\[
	\lambda_k(x)\in[0,\lambda_*].
	\]
	Moreover,
	\[
	0\leq\lambda_k(x)
	\leq
	\frac43h_k(x)
	\leq
	\frac23\varepsilon_k|x|^2.
	\]
	
	For $0\leq\lambda<\lambda_k(x)$ one has
	\[
	F_{k,x}(\lambda)>0,
	\]
	so the corresponding point of the normal ray lies locally above the graph
	of $f_{k+1}$ and therefore outside $C_{k+1}$.  Hence the zero just found is
	the first intersection with the next body.
	
	Finally,
	\[
	\left|
	\lambda_k(x)f_k'(x)
	\right|
	\leq
	C\varepsilon_k|x|^3,
	\]
	because $|f_k'(x)|\leq M_2|x|$.
\end{proof}
For $|x|\leq r_0$, define
\[
\mathcal G_k(c_k(x))
:=
\bigl(
x+\lambda_k(x)f_k'(x),
f_k(x)-\lambda_k(x)
\bigr)
\in\partial C_{k+1}.
\]
Equivalently, if $\tau_k$ denotes the corresponding parameter for the unit
normal, then
\[
\mathcal G_k(c)=c-\tau_k(c)\nu_k(c).
\]
In particular,
\[
\mathcal G_k(p)=p .
\]

% ---------------------------------------------------------------
\subsection{Offset independence}
\label{subsec:offset-independence}
% ---------------------------------------------------------------

Suppose now that an outer boundary is obtained from \(C_k\) by a
parallel offset of distance \(d_k>0\):
\begin{equation}
\Phi_{k,d_k}(c)
=
c+d_k\nu_k(c).
\label{eq:parallel-offset-map}
\end{equation}
At the corresponding point of the offset boundary, the inward unit
normal is
\[
-\nu_k(c).
\]

Starting from \(\Phi_{k,d_k}(c)\), the inward ray is therefore
\begin{equation}
\Phi_{k,d_k}(c)-t\nu_k(c),
\qquad
t\ge0.
\label{eq:offset-inward-ray}
\end{equation}

The following structural fact removes the offset distance completely
from the transition dynamics.

\begin{theorem}[Offset independence]
	\label{thm:offset-independence}
	Let \(c\in\partial C_k\) be sufficiently close to \(p\) so that
	Lemma~\ref{lem:uniform-normal-intersection} applies.
	Then the first intersection of the ray
	\eqref{eq:offset-inward-ray} with \(\partial C_{k+1}\) is exactly
	\(\mathcal G_k(c)\), independently of the offset distance \(d_k\).
	
	Equivalently, if \(\mathcal G_{k,d_k}\) denotes the transition
	constructed through the parallel offset, then
	\begin{equation}
	\boxed{
		\mathcal G_{k,d_k}=\mathcal G_k.
	}
	\label{eq:offset-independent-transition}
	\end{equation}
\end{theorem}

\begin{proof}
	Since \(C_k\) is convex, the tangent line to \(\partial C_k\) at
	\(c\) is supporting.  Hence
	\begin{equation}
	C_{k+1}
	\subset
	C_k
	\subset
	\{z\in\mathbb R^2:
	\langle z-c,\nu_k(c)\rangle\le0\}.
	\label{eq:supporting-halfspace}
	\end{equation}
	
	For \(0\le t<d_k\),
	\[
	\Phi_{k,d_k}(c)-t\nu_k(c)
	=
	c+(d_k-t)\nu_k(c),
	\]
	and therefore
	\[
	\left\langle
	\Phi_{k,d_k}(c)-t\nu_k(c)-c,
	\nu_k(c)
	\right\rangle
	=
	d_k-t>0.
	\]
	Such points lie strictly outside the supporting half-space containing
	\(C_{k+1}\).  Hence the offset ray cannot meet \(C_{k+1}\) before it
	reaches \(c\).
	
	For \(t=d_k+u\), with \(u\ge0\),
	\[
	\Phi_{k,d_k}(c)-t\nu_k(c)
	=
	c-u\nu_k(c).
	\]
	Thus, after passing through \(c\), the offset ray coincides exactly
	with the inward normal ray \(\mathcal R_k(c)\).  Its first
	intersection with \(C_{k+1}\) is consequently
	\(\mathcal G_k(c)\), independently of \(d_k\).
\end{proof}

\begin{remark}
	\label{rem:offset-ray-clarification}
	The statement is not that the entire half-line
	\[
	\{c+s\nu_k(c):s\le d_k\}
	\]
	is independent of \(d_k\); its outer endpoint clearly depends on the
	offset.  What is independent of \(d_k\) is the portion relevant to
	the first intersection with the nested body, namely
	\[
	\{c-t\nu_k(c):t\ge0\}.
	\]
\end{remark}

Theorem~\ref{thm:offset-independence} allows us from now on to work
directly with the consecutive boundaries \(C_k\) and \(C_{k+1}\).
The auxiliary parallel offset plays no role in the local transition
map or in any of its Taylor coefficients.

% ---------------------------------------------------------------
\subsection{Local intersection equation and signed arclength}
\label{subsec:local-intersection-arclength}
% ---------------------------------------------------------------

Because the normal line is unchanged by rescaling its direction
vector, it is convenient to use the unnormalised outward normal
\[
\widetilde\nu_k(x)
=
\bigl(-f_k'(x),1\bigr).
\]
Write
\begin{equation}
\mathcal G_k(c_k(x))
=
c_{k+1}(X_k(x)).
\label{eq:Xk-definition}
\end{equation}
Then there exists a scalar \(\lambda_k(x)\ge0\) such that
\begin{equation}
\bigl(X_k(x),f_{k+1}(X_k(x))\bigr)
=
\bigl(x,f_k(x)\bigr)
-
\lambda_k(x)\widetilde\nu_k(x).
\label{eq:unnormalized-normal-intersection}
\end{equation}
Consequently,
\begin{equation}
X_k(x)
=
x+\lambda_k(x)f_k'(x)
\label{eq:Xk-lambda}
\end{equation}
and
\begin{equation}
f_k(x)-\lambda_k(x)
=
f_{k+1}\!\left(
x+\lambda_k(x)f_k'(x)
\right).
\label{eq:exact-local-intersection}
\end{equation}

Equation~\eqref{eq:exact-local-intersection} is the exact local
equation of the geometric transition.

Define the geometric variation between two consecutive boundaries by
\begin{equation}
\varepsilon_k
:=
\|f_{k+1}-f_k\|_{C^2(I_{\rho_0})}.
\label{eq:geometric-variation}
\end{equation}

Moreover,
\[
0\leq\lambda_k(x)\leq C\varepsilon_k|x|^2,
\qquad
|X_k(x)-x|\leq C\varepsilon_k|x|^3.
\]

We now pass from graph coordinates to intrinsic coordinates on the
boundaries.  Define the signed arclength coordinate on
\(\partial C_k\) by
\begin{equation}
s_k(x)
:=
\int_0^x
\sqrt{1+f_k'(u)^2}\,du.
\label{eq:signed-arclength}
\end{equation}
Then
\[
s_k(0)=0,
\qquad
s_k'(0)=1,
\]
and \(s_k\) is a local diffeomorphism near \(0\).

The scalar transition map
\begin{equation}
G_k
:=
s_{k+1}\circ\mathcal G_k\circ s_k^{-1}
\label{eq:scalar-transition-map}
\end{equation}
is therefore characterized by
\begin{equation}
G_k(s_k(x))
=
s_{k+1}(X_k(x)).
\label{eq:scalar-transition-coordinate}
\end{equation}

In particular,
\[
G_k(0)=0.
\]

% ---------------------------------------------------------------
\subsection{Transition defect and diagonal vanishing}
\label{subsec:transition-defect}
% ---------------------------------------------------------------

The preceding geometry yields a uniform estimate for the entire
transition defect.

\begin{theorem}[Geometric transition-defect estimate]
	\label{thm:transition-defect}
	Under the assumptions of this section, there exist \(r>0\) and
	\(C>0\), independent of \(k\), such that
	\begin{equation}
	\boxed{
		|G_k(s)-s|
		\le
		C\varepsilon_k |s|^3,
		\qquad
		|s|\le r.
	}
	\label{eq:transition-defect}
	\end{equation}
\end{theorem}

\begin{proof}
	Let
	\[
	s=s_k(x),
	\qquad
	G_k(s)=s_{k+1}(X_k(x)).
	\]
	Then
	\begin{align}
	|G_k(s)-s|
	&\le
	|s_{k+1}(X_k(x))-s_{k+1}(x)|
	\nonumber\\
	&\quad+
	|s_{k+1}(x)-s_k(x)|.
	\label{eq:defect-splitting}
	\end{align}
	
	For the first term, the uniform \(C^2\) bound gives a uniform bound
	on \(s_{k+1}'\).  Hence Lemma~\ref{lem:uniform-normal-intersection}
	implies
	\begin{equation}
	|s_{k+1}(X_k(x))-s_{k+1}(x)|
	\le
	C\varepsilon_k|x|^3.
	\label{eq:first-defect-term}
	\end{equation}
	
	For the second term, set
	\[
	a_j(u):=\sqrt{1+f_j'(u)^2}.
	\]
	Then
	\[
	s_{k+1}(x)-s_k(x)
	=
	\int_0^x
	\bigl(a_{k+1}(u)-a_k(u)\bigr)\,du.
	\]
	Moreover,
	\begin{align}
	|a_{k+1}(u)-a_k(u)|
	&=
	\frac{
		|f_{k+1}'(u)-f_k'(u)|
		\,|f_{k+1}'(u)+f_k'(u)|
	}{
		a_{k+1}(u)+a_k(u)
	}.
	\label{eq:arclength-density-difference}
	\end{align}
	Because
	\[
	f_{k+1}'(0)-f_k'(0)=0,
	\]
	the definition of \(\varepsilon_k\) gives
	\[
	|f_{k+1}'(u)-f_k'(u)|
	\le
	\varepsilon_k|u|.
	\]
	The uniform \(C^2\) bound and \(f_j'(0)=0\) give
	\[
	|f_{k+1}'(u)+f_k'(u)|
	\le
	2M_2|u|.
	\]
	Thus
	\[
	|a_{k+1}(u)-a_k(u)|
	\le
	C\varepsilon_k |u|^2,
	\]
	and integration yields
	\begin{equation}
	|s_{k+1}(x)-s_k(x)|
	\le
	C\varepsilon_k|x|^3.
	\label{eq:second-defect-term}
	\end{equation}
	
	Combining
	\eqref{eq:defect-splitting},
	\eqref{eq:first-defect-term}, and
	\eqref{eq:second-defect-term},
	\[
	|G_k(s_k(x))-s_k(x)|
	\le
	C\varepsilon_k|x|^3.
	\]
	Finally,
	\[
	|s_k(x)|
	=
	\left|
	\int_0^x\sqrt{1+f_k'(u)^2}\,du
	\right|
	\ge |x|,
	\]
	so
	\[
	|x|^3\le |s_k(x)|^3.
	\]
	After reducing the common neighbourhood if necessary, this gives
	\eqref{eq:transition-defect}.
\end{proof}

The estimate has two immediate consequences.

\begin{corollary}[Parabolic tangency]
	\label{cor:parabolic-tangency}
	Assume in addition that the boundaries are \(C^3\).  Then
	\begin{equation}
	G_k(0)=0,
	\qquad
	G_k'(0)=1,
	\qquad
	G_k''(0)=0.
	\label{eq:parabolic-2jet}
	\end{equation}
	Thus the first possible nontrivial term in the local transition is
	cubic.
\end{corollary}

\begin{proof}
	Equation~\eqref{eq:transition-defect} gives
	\[
	G_k(s)-s=O(|s|^3)
	\]
	as \(s\to0\).  Since \(G_k\) is \(C^2\), its constant, linear, and
	quadratic Taylor coefficients are therefore those of the identity.
\end{proof}

\begin{corollary}[Diagonal vanishing]
	\label{cor:diagonal-vanishing}
	If
	\[
	f_{k+1}=f_k
	\]
	on the local chart, then
	\begin{equation}
	\boxed{
		G_k=\operatorname{Id}.
	}
	\label{eq:diagonal-identity}
	\end{equation}
	Consequently, whenever the boundaries have sufficient regularity,
	every nonidentity Taylor coefficient of \(G_k\) vanishes on the
	diagonal
	\[
	f_{k+1}=f_k.
	\]
\end{corollary}

This diagonal-vanishing principle is structural: the departure of
\(G_k\) from the identity is generated by the geometric difference
between two consecutive boundaries, rather than by the auxiliary
parallel offset.

Finally, define the normalized transition defect
\begin{equation}
D_k(r)
:=
\sup_{0<|s|\le r}
\frac{|G_k(s)-s|}{|s|^3}.
\label{eq:normalized-transition-defect}
\end{equation}
Theorem~\ref{thm:transition-defect} gives
\begin{equation}
D_k(r)\le C\varepsilon_k.
\label{eq:defect-controlled-by-variation}
\end{equation}
Hence the geometric summability condition
\begin{equation}
\sum_{k=0}^{\infty}\varepsilon_k<\infty
\label{eq:summable-geometric-variation}
\end{equation}
implies
\begin{equation}
\boxed{
	\sum_{k=0}^{\infty}D_k(r)<\infty.
}
\label{eq:summable-transition-defect}
\end{equation}

Thus summable variation of consecutive boundaries controls the
\emph{entire} transition defect.  No separate summability assumption
on finitely many Taylor coefficients, and no uncontrolled
higher-order remainder, is required at this stage.  This estimate
will provide the geometric basis for the non-autonomous convergence
theory developed later in the paper.
% ================================================================
\section{Local Parabolic Normal Form}
\label{sec:local-parabolic-normal-form}
% ================================================================

We now refine the geometric transition estimate of
Section~\ref{sec:geometric-transition} by computing the first
nontrivial Taylor coefficients of the scalar transition map
\[
G_k
=
s_{k+1}\circ\mathcal G_k\circ s_k^{-1}
\]
at the common tangency point.

Write
\begin{equation}
\kappa_k:=-f_k''(0).
\label{eq:kappa-definition}
\end{equation}
With the convention of Section~\ref{sec:geometric-transition},
convexity gives
\[
\kappa_k\ge0.
\]
Moreover, since
\[
f_{k+1}\le f_k,
\qquad
f_{k+1}(0)=f_k(0),
\qquad
f_{k+1}'(0)=f_k'(0)=0,
\]
the function \(f_k-f_{k+1}\) has a local minimum at the origin, and
hence
\begin{equation}
\kappa_{k+1}\ge\kappa_k.
\label{eq:curvature-monotonicity}
\end{equation}

% ---------------------------------------------------------------
\subsection{The cubic tangency law}
\label{subsec:cubic-tangency-law}
% ---------------------------------------------------------------

The first possible departure from the identity is cubic, as already
shown abstractly by Corollary~\ref{cor:parabolic-tangency}.  Its
coefficient depends only on the two curvatures at the common
tangency point.

\begin{theorem}[Cubic tangency law]
	\label{thm:cubic-tangency-law}
	Assume that \(f_k\) and \(f_{k+1}\) are \(C^4\) near the origin.
	Then
	\begin{equation}
	G_k(s)
	=
	s+\beta_k s^3+O(s^4),
	\qquad
	s\to0,
	\label{eq:cubic-normal-form}
	\end{equation}
	where
	\begin{equation}
	\boxed{
		\beta_k
		=
		\frac{
			(\kappa_{k+1}-\kappa_k)
			(\kappa_{k+1}-2\kappa_k)
		}{6}.
	}
	\label{eq:cubic-coefficient}
	\end{equation}
	In particular, \(\beta_k\) is independent of any auxiliary parallel
	offset.
\end{theorem}

\begin{proof}
	Write
	\[
	\Delta\kappa_k:=\kappa_{k+1}-\kappa_k.
	\]
	Near the origin,
	\[
	f_k(x)
	=
	-\frac{\kappa_k}{2}x^2+O(x^3),
	\qquad
	f_{k+1}(x)
	=
	-\frac{\kappa_{k+1}}{2}x^2+O(x^3).
	\]
	
	Let \(\lambda_k(x)\) and \(X_k(x)\) be defined by
	\eqref{eq:Xk-lambda}--\eqref{eq:exact-local-intersection}.  Matching
	the quadratic terms in
	\[
	f_k(x)-\lambda_k(x)
	=
	f_{k+1}\!\left(
	x+\lambda_k(x)f_k'(x)
	\right)
	\]
	gives
	\begin{equation}
	\lambda_k(x)
	=
	\frac{\Delta\kappa_k}{2}x^2+O(x^3).
	\label{eq:lambda-cubic-expansion}
	\end{equation}
	Since
	\[
	f_k'(x)
	=
	-\kappa_k x+O(x^2),
	\]
	we obtain
	\begin{equation}
	X_k(x)
	=
	x-\frac{\kappa_k\Delta\kappa_k}{2}x^3
	+O(x^4).
	\label{eq:X-cubic-expansion}
	\end{equation}
	
	The signed arclength coordinate on \(\partial C_j\) satisfies
	\begin{equation}
	s_j(x)
	=
	x+\frac{\kappa_j^2}{6}x^3+O(x^4).
	\label{eq:arclength-cubic-expansion}
	\end{equation}
	Hence
	\[
	s_k^{-1}(s)
	=
	s-\frac{\kappa_k^2}{6}s^3+O(s^4).
	\]
	Using
	\[
	G_k(s_k(x))
	=
	s_{k+1}(X_k(x)),
	\]
	we find
	\[
	G_k(s)
	=
	s+
	\left[
	-\frac{\kappa_k\Delta\kappa_k}{2}
	+
	\frac{\kappa_{k+1}^2-\kappa_k^2}{6}
	\right]s^3
	+
	O(s^4).
	\]
	Since
	\[
	\kappa_{k+1}^2-\kappa_k^2
	=
	\Delta\kappa_k(\kappa_{k+1}+\kappa_k),
	\]
	the coefficient reduces to
	\[
	\frac{
		\Delta\kappa_k(\kappa_{k+1}-2\kappa_k)
	}{6},
	\]
	which is \eqref{eq:cubic-coefficient}.
\end{proof}

\begin{remark}[Sign of the cubic coefficient]
	\label{rem:cubic-sign}
	Under the nesting relation
	\(\kappa_{k+1}\ge\kappa_k\), one has
	\[
	\beta_k<0
	\quad\Longleftrightarrow\quad
	\kappa_k<\kappa_{k+1}<2\kappa_k,
	\]
	whereas
	\[
	\beta_k>0
	\quad\Longleftrightarrow\quad
	\kappa_{k+1}>2\kappa_k.
	\]
	The cubic coefficient vanishes when either
	\[
	\kappa_{k+1}=\kappa_k
	\]
	or
	\[
	\kappa_{k+1}=2\kappa_k.
	\]
	The first vanishing is caused by equality of the quadratic boundary
	jets; the second is a genuine cubic degeneracy.
\end{remark}

% ---------------------------------------------------------------
\subsection{The non-symmetric quartic term}
\label{subsec:quartic-term}
% ---------------------------------------------------------------

The cubic coefficient depends only on curvature.  The first
coefficient detecting asymmetry of the boundaries appears at fourth
order.

Write
\begin{equation}
f_k(x)
=
-\frac{\kappa_k}{2}x^2
+
\frac{a_k}{6}x^3
+
O(x^4),
\qquad
a_k:=f_k'''(0).
\label{eq:boundary-cubic-jet}
\end{equation}

\begin{proposition}[Quartic coefficient]
	\label{prop:quartic-coefficient}
	Assume that \(f_k,f_{k+1}\in C^5\).  Then
	\begin{equation}
	G_k(s)
	=
	s+\beta_k s^3+\delta_k s^4+O(s^5),
	\label{eq:quartic-normal-form}
	\end{equation}
	where \(\beta_k\) is given by
	\eqref{eq:cubic-coefficient} and
	\begin{equation}
	\boxed{
		\delta_k
		=
		\frac{
			3a_k(\kappa_{k+1}-\kappa_k)
			+
			(a_{k+1}-a_k)(4\kappa_k-3\kappa_{k+1})
		}{24}.
	}
	\label{eq:quartic-coefficient}
	\end{equation}
\end{proposition}

\begin{proof}
	Set
	\[
	\Delta\kappa_k
	=
	\kappa_{k+1}-\kappa_k,
	\qquad
	\Delta a_k
	=
	a_{k+1}-a_k.
	\]
	Solving the exact intersection equation
	\eqref{eq:exact-local-intersection} one order further gives
	\begin{equation}
	\lambda_k(x)
	=
	\frac{\Delta\kappa_k}{2}x^2
	-
	\frac{\Delta a_k}{6}x^3
	+
	O(x^4).
	\label{eq:lambda-quartic-expansion}
	\end{equation}
	Since
	\[
	f_k'(x)
	=
	-\kappa_kx+\frac{a_k}{2}x^2+O(x^3),
	\]
	equation~\eqref{eq:Xk-lambda} yields
	\begin{equation}
	X_k(x)
	=
	x
	-
	\frac{\kappa_k\Delta\kappa_k}{2}x^3
	+
	\left(
	\frac{\kappa_k\Delta a_k}{6}
	+
	\frac{a_k\Delta\kappa_k}{4}
	\right)x^4
	+
	O(x^5).
	\label{eq:X-quartic-expansion}
	\end{equation}
	
	The signed arclength expansion is
	\begin{equation}
	s_k(x)
	=
	x
	+
	\frac{\kappa_k^2}{6}x^3
	-
	\frac{\kappa_ka_k}{8}x^4
	+
	O(x^5).
	\label{eq:arclength-quartic-expansion}
	\end{equation}
	Substitution into
	\[
	G_k(s_k(x))
	=
	s_{k+1}(X_k(x))
	\]
	and collection of fourth-order terms gives
	\[
	\delta_k
	=
	\frac{\kappa_ka_k}{8}
	+
	\frac{\kappa_k\Delta a_k}{6}
	+
	\frac{a_k\Delta\kappa_k}{4}
	-
	\frac{\kappa_{k+1}a_{k+1}}{8}.
	\]
	Elementary simplification gives
	\eqref{eq:quartic-coefficient}.
\end{proof}

An equivalent form is
\begin{equation}
\delta_k
=
\frac{
	a_k(6\kappa_{k+1}-7\kappa_k)
	+
	a_{k+1}(4\kappa_k-3\kappa_{k+1})
}{24}.
\label{eq:quartic-coefficient-equivalent}
\end{equation}

If both boundaries are locally symmetric with respect to the common
normal line, then
\[
a_k=a_{k+1}=0,
\]
and therefore
\[
\delta_k=0.
\]

% ---------------------------------------------------------------
\subsection{Symmetry, higher-order coefficients, and diagonal vanishing}
\label{subsec:higher-order-diagonal}
% ---------------------------------------------------------------

Assume now that both local boundary graphs are even:
\begin{equation}
f_j(-x)=f_j(x),
\qquad
j=k,k+1.
\label{eq:local-even-symmetry}
\end{equation}
Reflection in the common normal line preserves both boundaries and
their normal directions.  By uniqueness of the local normal
intersection,
\[
X_k(-x)=-X_k(x).
\]
Since the signed arclength coordinates are odd,
\[
s_j(-x)=-s_j(x),
\]
the scalar transition map satisfies
\begin{equation}
\boxed{
	G_k(-s)=-G_k(s).
}
\label{eq:Gk-odd}
\end{equation}

Thus every even Taylor coefficient vanishes.

\begin{proposition}[Symmetric parabolic normal form]
	\label{prop:symmetric-normal-form}
	Under the local symmetry assumption and sufficient regularity,
	\begin{equation}
	G_k(s)
	=
	s+\beta_k s^3+\gamma_k s^5+R_k(s),
	\label{eq:symmetric-parabolic-normal-form}
	\end{equation}
	where \(\beta_k\) is given by
	\eqref{eq:cubic-coefficient}.
	
	If the transition maps are \(C^6\), then
	\[
	R_k(s)=o(|s|^6).
	\]
	If, in addition, their seventh derivatives are uniformly bounded on
	a common neighbourhood of the origin, then
	\begin{equation}
	|R_k(s)|
	\le
	C|s|^7
	\label{eq:uniform-seventh-remainder}
	\end{equation}
	with \(C\) independent of \(k\).
\end{proposition}

The remainder estimate
\eqref{eq:uniform-seventh-remainder} is a local Taylor statement.  It
will not be used by itself to control an infinite non-autonomous
orbit.  The geometric convergence theory will instead use the full
transition-defect estimate
\eqref{eq:transition-defect}.

The preceding explicit formulas are instances of a more general
structural fact.

For \(r\ge2\), denote by
\[
J_0^r f
=
\bigl(
f''(0),f'''(0),\ldots,f^{(r)}(0)
\bigr)
\]
the nontrivial \(r\)-jet of a boundary graph at the common tangency
point.

\begin{proposition}[Smooth jet dependence and diagonal vanishing]
	\label{prop:jet-diagonal-vanishing}
	For every Taylor coefficient of \(G_k-\operatorname{Id}\) that is
	defined by the available regularity, there exists a smooth function
	of finitely many jets of \(f_k\) and \(f_{k+1}\) that produces that
	coefficient.
	
	Moreover, this function vanishes when the two corresponding jets are
	equal.  Consequently, on any uniformly bounded jet set,
	\begin{equation}
	|\alpha_{k,j}|
	\le
	C_j
	\left|
	J_0^{r_j}f_{k+1}
	-
	J_0^{r_j}f_k
	\right|,
	\label{eq:coefficient-jet-bound}
	\end{equation}
	where \(\alpha_{k,j}\) denotes the coefficient of \(s^j\) in
	\(G_k(s)-s\), and \(r_j\) is some finite jet order sufficient to
	determine that coefficient.
\end{proposition}

\begin{proof}
	The local transition is determined by the implicit equation
	\[
	f_k(x)-\lambda
	-
	f_{k+1}\!\left(
	x+\lambda f_k'(x)
	\right)
	=
	0.
	\]
	At the common tangency point,
	\[
	\partial_\lambda
	\left[
	f_k(x)-\lambda
	-
	f_{k+1}\!\left(
	x+\lambda f_k'(x)
	\right)
	\right]_{(x,\lambda)=(0,0)}
	=
	-1.
	\]
	The implicit-function theorem therefore shows that the Taylor
	coefficients of \(\lambda_k(x)\), and hence those of \(X_k(x)\), are
	smooth functions of finitely many boundary jets.  The same is true
	after the smooth changes to signed arclength coordinates.
	
	If the two jets coincide to the required order, the corresponding
	Taylor coefficient is the same as on the diagonal
	\(f_{k+1}=f_k\).  By
	Corollary~\ref{cor:diagonal-vanishing}, the transition is then the
	identity to that order, so the nonidentity coefficient vanishes.
	The estimate \eqref{eq:coefficient-jet-bound} follows from the
	mean-value theorem on a bounded jet set.
\end{proof}

In particular,
\begin{equation}
|\beta_k|
\le
C|\kappa_{k+1}-\kappa_k|,
\label{eq:beta-variation-bound}
\end{equation}
and
\begin{equation}
|\delta_k|
\le
C\left(
|\kappa_{k+1}-\kappa_k|
+
|a_{k+1}-a_k|
\right).
\label{eq:delta-variation-bound}
\end{equation}

Under local symmetry, the same principle gives
\begin{equation}
|\gamma_k|
\le
C
\left|
J_0^{r_\gamma}f_{k+1}
-
J_0^{r_\gamma}f_k
\right|
\label{eq:gamma-jet-bound}
\end{equation}
for a finite jet order \(r_\gamma\).

Consequently, if the boundary jets converge, then
\begin{equation}
\beta_k\to0,
\qquad
\delta_k\to0,
\qquad
\gamma_k\to0
\label{eq:coefficients-vanish}
\end{equation}
whenever the corresponding coefficients are defined.

This conclusion is purely geometric: consecutive boundaries that
become indistinguishable at the relevant jet order produce
transition maps whose corresponding nonidentity coefficients vanish.
\begin{proposition}[Absolute summability of the cubic coefficients]
	\label{prop:cubic-summability}
	Under the standing hypotheses of Section~2, assume that the boundaries are
	sufficiently regular for the cubic coefficients $\beta_k$ to be defined.
	Then
	\[
	\sum_{k=0}^{\infty} |\beta_k| < \infty .
	\]
	More precisely, if
	\[
	\kappa_k=-f_k''(0),
	\]
	then
	\[
	0\leq \kappa_k\leq M_2,
	\qquad
	\kappa_{k+1}\geq \kappa_k,
	\]
	and
	\[
	|\beta_k|
	\leq
	\frac{M_2}{2}\,
	(\kappa_{k+1}-\kappa_k).
	\]
	Consequently,
	\[
	\sum_{k=0}^{\infty}|\beta_k|
	\leq
	\frac{M_2}{2}
	(\kappa_\infty-\kappa_0)
	<\infty,
	\]
	where
	\[
	\kappa_\infty:=\lim_{k\to\infty}\kappa_k.
	\]
\end{proposition}

\begin{proof}
	By nesting and the common tangency conditions,
	\[
	\kappa_{k+1}\geq\kappa_k,
	\]
	while the uniform $C^2$ bound of Section~2 gives
	\[
	0\leq\kappa_k\leq M_2.
	\]
	Hence $(\kappa_k)$ is monotone and bounded and therefore converges to some
	$\kappa_\infty\leq M_2$.
	
	By the cubic tangency law,
	\[
	\beta_k
	=
	\frac{
		(\kappa_{k+1}-\kappa_k)
		(\kappa_{k+1}-2\kappa_k)
	}{6}.
	\]
	Since $\kappa_{k+1}-\kappa_k\geq0$,
	\[
	\begin{aligned}
	|\beta_k|
	&\leq
	\frac{\kappa_{k+1}-\kappa_k}{6}
	\left(
	\kappa_{k+1}+2\kappa_k
	\right)
	\\
	&\leq
	\frac{M_2}{2}
	(\kappa_{k+1}-\kappa_k).
	\end{aligned}
	\]
	Summing and using telescoping gives
	\[
	\sum_{k=0}^{N}|\beta_k|
	\leq
	\frac{M_2}{2}
	(\kappa_{N+1}-\kappa_0).
	\]
	Letting $N\to\infty$ proves the assertion.
\end{proof}

\begin{remark}[No nonsummable cubic drift from nested curvature]
	\label{rem:no-critical-cubic}
	Proposition~\ref{prop:cubic-summability} requires no summability assumption
	on the full geometric variation
	\[
	\|f_{k+1}-f_k\|_{C^2}.
	\]
	Absolute summability of the cubic coefficients follows already from nesting
	and the uniform curvature bound.
	
	In particular, a genuinely nonsummable cubic regime such as
	\[
	-\beta_k\sim\frac{c}{k},
	\qquad c>0,
	\]
	cannot be generated by the cubic curvature term of a nested family satisfying
	the standing hypotheses.  Any geometrically realizable critical regime must
	therefore arise from higher-order jet variation, from the complete transition
	defect, or from a setting outside these hypotheses.
\end{remark}
% ---------------------------------------------------------------
\subsection{The parabolic local model}
\label{subsec:parabolic-local-model}
% ---------------------------------------------------------------

The family of quadratic graphs provides an explicit model in which
the quintic coefficient can be computed in closed form.  Let
\begin{equation}
f_k(x)
=
-\frac{\kappa_k}{2}x^2,
\qquad
f_{k+1}(x)
=
-\frac{\kappa_{k+1}}{2}x^2.
\label{eq:parabolic-boundaries}
\end{equation}
The boundaries are locally symmetric, so
\[
G_k(s)
=
s+\beta_k s^3+\gamma_k s^5+O(s^7).
\]

\begin{proposition}[Explicit quintic coefficient for the parabolic model]
	\label{prop:parabolic-quintic}
	For the quadratic graphs \eqref{eq:parabolic-boundaries},
	\begin{equation}
	\boxed{
		\gamma_k
		=
		\frac{
			(\kappa_{k+1}-\kappa_k)
			Q(\kappa_k,\kappa_{k+1})
		}{120},
	}
	\label{eq:parabolic-gamma}
	\end{equation}
	where
	\begin{equation}
	Q(a,b)
	=
	-3b^3
	-33ab^2
	+47a^2b
	+17a^3.
	\label{eq:parabolic-Q}
	\end{equation}
\end{proposition}

\begin{proof}
	For the quadratic graphs, the exact intersection equation becomes
	\[
	-\frac{\kappa_k}{2}x^2-\lambda
	=
	-\frac{\kappa_{k+1}}{2}
	\bigl(x-\kappa_kx\lambda\bigr)^2.
	\]
	Solving for \(\lambda\) to fourth order gives
	\[
	\lambda
	=
	\frac{\kappa_{k+1}-\kappa_k}{2}x^2
	-
	\frac{
		\kappa_k\kappa_{k+1}
		(\kappa_{k+1}-\kappa_k)
	}{2}x^4
	+
	O(x^6),
	\]
	and therefore
	\[
	X_k(x)
	=
	x
	-
	\frac{
		\kappa_k(\kappa_{k+1}-\kappa_k)
	}{2}x^3
	+
	\frac{
		\kappa_k^2\kappa_{k+1}
		(\kappa_{k+1}-\kappa_k)
	}{2}x^5
	+
	O(x^7).
	\]
	
	For a quadratic graph of curvature parameter \(\kappa\),
	\begin{equation}
	s_\kappa(x)
	=
	x
	+
	\frac{\kappa^2}{6}x^3
	-
	\frac{\kappa^4}{40}x^5
	+
	O(x^7).
	\label{eq:parabolic-arclength}
	\end{equation}
	Substitution into
	\[
	G_k(s_k(x))
	=
	s_{k+1}(X_k(x))
	\]
	and collection of fifth-order terms gives
	\[
	120\gamma_k
	=
	-3\kappa_{k+1}^4
	-30\kappa_k\kappa_{k+1}^3
	+80\kappa_k^2\kappa_{k+1}^2
	-30\kappa_k^3\kappa_{k+1}
	-17\kappa_k^4.
	\]
	The right-hand side factors as
	\[
	(\kappa_{k+1}-\kappa_k)
	Q(\kappa_k,\kappa_{k+1}),
	\]
	which proves \eqref{eq:parabolic-gamma}.
\end{proof}

If \(\kappa_k>0\) and
\[
r_k:=\frac{\kappa_{k+1}}{\kappa_k},
\]
then
\begin{equation}
\gamma_k
=
\frac{\kappa_k^4}{120}
(r_k-1)
\bigl(
-3r_k^3-33r_k^2+47r_k+17
\bigr).
\label{eq:parabolic-gamma-ratio}
\end{equation}

The factor \(r_k-1\) is the explicit manifestation of diagonal
vanishing.  In particular,
\[
\kappa_{k+1}-\kappa_k\to0
\]
with uniformly bounded curvatures implies
\[
\gamma_k\to0.
\]

The polynomial
\[
-3r^3-33r^2+47r+17
\]
has a positive zero near
\[
r\simeq1.542.
\]
Thus the quintic coefficient can also vanish away from the diagonal.
This secondary cancellation is specific to the parabolic model and
is distinct from the structural diagonal vanishing
\(\kappa_{k+1}=\kappa_k\).

The cubic and quintic formulas of this subsection will later provide
explicit test families.  The general geometric dynamics, however,
will be based on the full transition estimate
\eqref{eq:transition-defect}, rather than on truncating the local
normal form.

% ================================================================
\section{Summable Non-Autonomous Geometry}
\label{sec:summable-geometry}
% ================================================================

The local normal form of Section~\ref{sec:local-parabolic-normal-form}
describes the first Taylor coefficients of each transition map.
For the long-time non-autonomous dynamics, however, it is preferable
to control the complete transition defect rather than a finite
Taylor truncation.

We therefore return to the geometric variation
\[
\varepsilon_k
=
\|f_{k+1}-f_k\|_{C^2(I_{\rho_0})}
\]
introduced in \eqref{eq:geometric-variation}, and to the normalized
transition defect \(D_k(r)\) of
\eqref{eq:normalized-transition-defect}.

% ---------------------------------------------------------------
\subsection{Summable geometric variation}
\label{subsec:summable-geometric-variation}
% ---------------------------------------------------------------

We impose the following geometric hypothesis.

\begin{definition}[Summable geometric variation]
	\label{def:summable-geometric-variation}
	The nested family has summable local geometric variation if
	\begin{equation}
	\boxed{
		\sum_{k=0}^{\infty}\varepsilon_k
		=
		\sum_{k=0}^{\infty}
		\|f_{k+1}-f_k\|_{C^2(I_{\rho_0})}
		<\infty.
	}
	\label{eq:summable-C2-variation}
	\end{equation}
\end{definition}

This is a bounded-variation condition on the sequence of local
boundary graphs.  It is stronger than mere \(C^2\) convergence:
convergence of \(f_k\) does not by itself imply
\eqref{eq:summable-C2-variation}.

The assumption has an immediate geometric consequence.

\begin{proposition}[Existence of a limiting boundary jet]
	\label{prop:limiting-boundary}
	Under \eqref{eq:summable-C2-variation}, there exists
	\[
	f_\infty\in C^2(I_{\rho_0})
	\]
	such that
	\begin{equation}
	f_k\longrightarrow f_\infty
	\qquad
	\text{in }C^2(I_{\rho_0}).
	\label{eq:C2-limit-boundary}
	\end{equation}
	Moreover,
	\begin{equation}
	\|f_\infty-f_k\|_{C^2}
	\le
	\sum_{j=k}^{\infty}\varepsilon_j.
	\label{eq:C2-tail-bound}
	\end{equation}
\end{proposition}

\begin{proof}
	For \(m>k\),
	\[
	f_m-f_k
	=
	\sum_{j=k}^{m-1}(f_{j+1}-f_j),
	\]
	and hence
	\[
	\|f_m-f_k\|_{C^2}
	\le
	\sum_{j=k}^{m-1}\varepsilon_j.
	\]
	Thus \((f_k)\) is Cauchy in \(C^2(I_{\rho_0})\).  Completeness gives
	the limit \(f_\infty\), and letting \(m\to\infty\) gives
	\eqref{eq:C2-tail-bound}.
\end{proof}

Because the common tangency conditions hold for every \(k\),
\[
f_\infty(0)=0,
\qquad
f_\infty'(0)=0.
\]
Thus the limiting boundary retains the common tangency point and
tangent direction.

% ---------------------------------------------------------------
\subsection{Summability of the transition defect}
\label{subsec:summable-transition-defect}
% ---------------------------------------------------------------

The geometric transition estimate
\eqref{eq:defect-controlled-by-variation} gives, on a fixed
neighbourhood \(|s|\le r\),
\[
D_k(r)\le C\varepsilon_k.
\]
Consequently:

\begin{proposition}[Summable transition defect]
	\label{prop:summable-transition-defect}
	If the geometric variation is summable, then
	\begin{equation}
	\boxed{
		A_r
		:=
		\sum_{k=0}^{\infty}D_k(r)
		<\infty.
	}
	\label{eq:Ar-definition}
	\end{equation}
	More precisely,
	\begin{equation}
	A_r
	\le
	C\sum_{k=0}^{\infty}\varepsilon_k.
	\label{eq:Ar-geometric-bound}
	\end{equation}
\end{proposition}

Thus
\begin{equation}
|G_k(s)-s|
\le
D_k(r)|s|^3,
\qquad
|s|\le r,
\label{eq:full-defect-bound}
\end{equation}
with a summable sequence of coefficients \(D_k(r)\).

The important point is that
\eqref{eq:full-defect-bound} controls the \emph{whole} transition
map.  No decomposition into cubic, quartic, quintic, and remainder
terms is required for the convergence argument.

Higher regularity gives a compatible coefficient-level statement.
For example, the estimates of
Section~\ref{subsec:higher-order-diagonal} imply that summable
variation of the corresponding boundary jets yields
\[
\sum_k|\beta_k|<\infty,
\qquad
\sum_k|\delta_k|<\infty,
\qquad
\sum_k|\gamma_k|<\infty
\]
whenever the relevant jet variations are summable.  These are
consequences of the geometric variation, rather than substitutes for
control of the full transition defect.

% ---------------------------------------------------------------
\subsection{Convergence to a nonzero limit}
\label{subsec:nonzero-limit}
% ---------------------------------------------------------------

We first state the dynamical argument at the level of the transition
defect.

\begin{theorem}[Finite drift under summable parabolic defects]
	\label{thm:finite-drift}
	Let
	\[
	G_k:(-r,r)\longrightarrow\mathbb R,
	\qquad
	G_k(0)=0,
	\]
	satisfy
	\begin{equation}
	|G_k(s)-s|
	\le
	a_k|s|^3,
	\qquad
	|s|\le r,
	\label{eq:abstract-summable-defect}
	\end{equation}
	where
	\begin{equation}
	A:=\sum_{k=0}^{\infty}a_k<\infty.
	\label{eq:abstract-A}
	\end{equation}
	
	Choose \(\rho>0\) such that
	\begin{equation}
	2\rho\le r,
	\qquad
	16A\rho^2\le1.
	\label{eq:small-initial-radius}
	\end{equation}
	Then for every initial value
	\[
	0<|s_0|\le\rho,
	\]
	the non-autonomous orbit
	\begin{equation}
	s_{k+1}=G_k(s_k)
	\label{eq:nonautonomous-orbit}
	\end{equation}
	is defined for all \(k\), remains in the local region, and satisfies
	\begin{equation}
	\boxed{
		\sum_{k=0}^{\infty}|s_{k+1}-s_k|
		\le
		8A|s_0|^3.
	}
	\label{eq:total-variation-orbit}
	\end{equation}
	
	Moreover,
	\begin{equation}
	|s_k-s_0|
	\le
	8A|s_0|^3
	\le
	\frac12|s_0|
	\qquad
	(k\ge0),
	\label{eq:uniform-orbit-drift}
	\end{equation}
	and therefore
	\begin{equation}
	\frac12|s_0|
	\le
	|s_k|
	\le
	\frac32|s_0|.
	\label{eq:orbit-annulus}
	\end{equation}
	
	In particular,
	\begin{equation}
	\boxed{
		s_k\longrightarrow s_\infty,
		\qquad
		s_\infty\neq0,
	}
	\label{eq:nonzero-limit}
	\end{equation}
	and
	\begin{equation}
	|s_\infty-s_0|
	\le
	8A|s_0|^3.
	\label{eq:limit-drift}
	\end{equation}
	The sign of \(s_k\) is the same as the sign of \(s_0\) for every
	\(k\), including the limit.
\end{theorem}

\begin{proof}
	Set
	\[
	a:=|s_0|.
	\]
	We first show that
	\begin{equation}
	|s_k|\le2a
	\qquad
	\text{for all }k.
	\label{eq:bootstrap-2a}
	\end{equation}
	
	Suppose otherwise, and let \(n\ge1\) be the first index for which
	\[
	|s_n|>2a.
	\]
	Then
	\[
	|s_j|\le2a\le2\rho\le r,
	\qquad
	0\le j<n.
	\]
	Using \eqref{eq:abstract-summable-defect},
	\begin{align}
	|s_n-s_0|
	&\le
	\sum_{j=0}^{n-1}|s_{j+1}-s_j|
	\nonumber\\
	&\le
	\sum_{j=0}^{n-1}a_j|s_j|^3
	\nonumber\\
	&\le
	8a^3\sum_{j=0}^{n-1}a_j
	\nonumber\\
	&\le
	8Aa^3.
	\label{eq:bootstrap-drift}
	\end{align}
	By \eqref{eq:small-initial-radius},
	\[
	8Aa^3\le\frac a2.
	\]
	Hence
	\[
	|s_n|
	\le
	|s_0|+|s_n-s_0|
	\le
	\frac32a<2a,
	\]
	contradicting the definition of \(n\).  Thus
	\eqref{eq:bootstrap-2a} holds.
	
	Since every iterate remains in the region of validity,
	\[
	|s_{k+1}-s_k|
	\le
	a_k(2a)^3
	=
	8a_k a^3.
	\]
	Summing over \(k\) yields
	\eqref{eq:total-variation-orbit}.  Consequently,
	\[
	|s_k-s_0|
	\le
	8Aa^3
	\le
	\frac a2,
	\]
	which proves \eqref{eq:uniform-orbit-drift} and
	\eqref{eq:orbit-annulus}.
	
	The series
	\[
	\sum_k|s_{k+1}-s_k|
	\]
	converges, so \((s_k)\) is Cauchy and therefore converges to some
	\(s_\infty\).  Passing to the limit in
	\eqref{eq:orbit-annulus} gives
	\[
	|s_\infty|\ge\frac12|s_0|>0.
	\]
	Moreover,
	\[
	|s_k-s_0|\leq \frac12|s_0|
	\]
	implies
	\[
	s_k s_0>0
	\]
	for every $k$.  Hence no iterate changes sign, and the same holds for the
	limit $s_\infty$.
	Finally, \eqref{eq:limit-drift} follows from
	\eqref{eq:total-variation-orbit}.
\end{proof}

Applying this theorem to the geometric maps gives the main result of
the section.

\begin{theorem}[Convergence under summable geometric variation]
	\label{thm:geometric-summable-convergence}
	Assume the geometric hypotheses of
	Section~\ref{sec:geometric-transition} and
	\begin{equation}
	\sum_{k=0}^{\infty}
	\|f_{k+1}-f_k\|_{C^2(I_{\rho_0})}
	<\infty.
	\label{eq:geometric-summability-theorem}
	\end{equation}
	Then there exists \(\rho>0\) such that every orbit
	\[
	s_{k+1}=G_k(s_k),
	\qquad
	0<|s_0|\le\rho,
	\]
	converges to a nonzero limit
	\[
	s_k\longrightarrow s_\infty,
	\qquad
	s_\infty\neq0.
	\]
	
	If
	\[
	E
	:=
	\sum_{k=0}^{\infty}\varepsilon_k,
	\]
	then, after choosing \(\rho\) so that
	\[
	2\rho\le r,
	\qquad
	16CE\rho^2\le1,
	\]
	one has
	\begin{equation}
	\boxed{
		\sum_{k=0}^{\infty}|s_{k+1}-s_k|
		\le
		8CE|s_0|^3
	}
	\label{eq:geometric-total-drift}
	\end{equation}
	and
	\begin{equation}
	\boxed{
		|s_\infty-s_0|
		\le
		8CE|s_0|^3.
	}
	\label{eq:geometric-limit-drift}
	\end{equation}
\end{theorem}

\begin{proof}
	By Theorem~\ref{thm:transition-defect},
	\[
	|G_k(s)-s|
	\le
	C\varepsilon_k|s|^3.
	\]
	Apply Theorem~\ref{thm:finite-drift} with
	\[
	a_k=C\varepsilon_k.
	\]
\end{proof}

\begin{remark}[No remainder summation]
	\label{rem:no-remainder-summation}
	Theorem~\ref{thm:geometric-summable-convergence} uses the complete
	geometric estimate
	\[
	|G_k(s)-s|
	\le
	C\varepsilon_k|s|^3.
	\]
	It therefore requires no separate summation of a Taylor remainder
	such as \(O(s^5)\) or \(O(s^7)\).  All higher-order contributions are
	already contained in the summable transition defect.
\end{remark}

% ---------------------------------------------------------------
\subsection{Lyapunov stability without asymptotic stability}
\label{subsec:lyapunov-not-asymptotic}
% ---------------------------------------------------------------

The common tangency point corresponds to the common fixed state
\[
s=0,
\qquad
G_k(0)=0
\quad
\text{for every }k.
\]

For the non-autonomous system
\[
s_{k+1}=G_k(s_k),
\]
we use the standard local notion of Lyapunov stability: the fixed
state \(0\) is stable if, for every \(\eta>0\), there exists
\(\delta>0\) such that
\[
|s_0|<\delta
\quad\Longrightarrow\quad
|s_k|<\eta
\quad
\text{for every }k\ge0.
\]

\begin{corollary}[Parabolic stability without attraction]
	\label{cor:stable-not-attractive}
	Under the hypotheses of
	Theorem~\ref{thm:geometric-summable-convergence}, the common
	tangency point is Lyapunov stable.
	
	It is not locally attractive: every sufficiently small nonzero
	initial condition satisfies
	\[
	s_k\longrightarrow s_\infty\neq0.
	\]
	Consequently the common tangency point is not asymptotically stable.
\end{corollary}

\begin{proof}
	By Theorem~\ref{thm:finite-drift},
	\[
	|s_k|
	\le
	\frac32|s_0|
	\]
	for all sufficiently small initial data.  Given \(\eta>0\), choose
	\[
	\delta
	<
	\min\left\{
	\rho,\frac{2\eta}{3}
	\right\}.
	\]
	Then
	\[
	|s_0|<\delta
	\quad\Longrightarrow\quad
	|s_k|<\eta
	\]
	for every \(k\), proving Lyapunov stability.
	
	For \(s_0\neq0\), the same theorem gives
	\[
	|s_\infty|
	\ge
	\frac12|s_0|>0.
	\]
	Thus nearby nonzero orbits do not converge to the fixed state.
\end{proof}

The accumulated non-autonomous drift can also be encoded by the
limiting map
\begin{equation}
\mathcal L(s_0)
:=
\lim_{k\to\infty}
G_{k-1}\circ\cdots\circ G_0(s_0),
\label{eq:limiting-map}
\end{equation}
defined for sufficiently small \(s_0\).

\begin{corollary}[Cubic asymptotics of the limiting displacement map]
	\label{cor:limiting-map-cubic}
	Under the hypotheses of Theorem~4.5, assume in addition that the
	boundaries have the regularity required in Theorem~3.1 so that the cubic
	coefficients $\beta_k$ are defined. Define then
	\[
	\mathcal L(s)
	:=
	\lim_{n\to\infty}
	G_{n-1}\circ\cdots\circ G_0(s)
	\]
	for sufficiently small $s$.
	Then
	\[
	\mathcal L(0)=0,
	\qquad
	\mathcal L'(0)=1,
	\]
	and
	\[
	\mathcal L(s)
	=
	s+B\,s^3+o(s^3),
	\qquad
	B:=\sum_{k=0}^{\infty}\beta_k.
	\]
	In particular, the cubic coefficient of the accumulated
	non-autonomous displacement is the sum of the cubic coefficients of the
	individual geometric transitions.
\end{corollary}

\begin{proof}
	For $N\geq1$, set
	\[
	F_N
	:=
	G_{N-1}\circ\cdots\circ G_0.
	\]
	Since every $G_k$ has the local expansion
	\[
	G_k(s)=s+\beta_k s^3+O(s^4),
	\]
	finite composition gives
	\[
	F_N(s)
	=
	s+B_Ns^3+O(s^4),
	\qquad
	B_N:=\sum_{k=0}^{N-1}\beta_k.
	\]
	Indeed, cubic coefficients add under composition because interaction terms
	between cubic corrections occur only at higher order.
	
	Let
	\[
	A_N:=\sum_{k=N}^{\infty}D_k(r).
	\]
	By Proposition~4.3,
	\[
	A_N\longrightarrow0.
	\]
	Applying the finite-drift estimate to the tail system beginning at time $N$
	gives, for sufficiently small $s$,
	\[
	|\mathcal L(s)-F_N(s)|
	\leq
	8A_N |F_N(s)|^3.
	\]
	Therefore
	\[
	\left|
	\frac{\mathcal L(s)-s}{s^3}
	-
	B_N
	\right|
	\leq
	\left|
	\frac{F_N(s)-s}{s^3}-B_N
	\right|
	+
	8A_N
	\left|
	\frac{F_N(s)}{s}
	\right|^3.
	\]
	For fixed $N$, letting $s\to0$ gives
	\[
	\limsup_{s\to0}
	\left|
	\frac{\mathcal L(s)-s}{s^3}
	-
	B_N
	\right|
	\leq
	8A_N,
	\]
	because $F_N'(0)=1$.
	Now let $N\to\infty$.  Since $A_N\to0$ and
	\[
	B_N\longrightarrow
	B=\sum_{k=0}^{\infty}\beta_k,
	\]
	we obtain
	\[
	\lim_{s\to0}
	\frac{\mathcal L(s)-s}{s^3}
	=B.
	\]
	Hence
	\[
	\mathcal L(s)
	=
	s+B s^3+o(s^3).
	\]
\end{proof}

Thus the non-autonomous composition has a well-defined limiting
effect near the common tangency point, but the accumulated map
remains tangent to the identity.  The local geometry produces a
finite nonlinear displacement rather than contraction to the common
fixed point.

% ================================================================
\section{Critical Abstract Recurrences}
\label{sec:abstract-recurrences}
% ================================================================

The geometric dynamics of Section~\ref{sec:summable-geometry} is
controlled by the complete transition defect and leads, under
summable geometric variation, to convergence toward a nonzero limit.
We now study a different problem.

This section concerns abstract scalar non-autonomous recurrences whose
effective drift is not summable.  These recurrences need not arise
from the convex geometry developed in the preceding sections.
Their purpose is to identify the parabolic asymptotic laws associated
with cubic, quintic, and quartic degeneracies under hypotheses that
control the full remainder at the relevant scale.

Throughout the section, the initial value \(s_0>0\) is chosen
sufficiently small.

% ---------------------------------------------------------------
\subsection{A parabolic comparison principle}
\label{subsec:parabolic-comparison}
% ---------------------------------------------------------------

We begin with a general recurrence of the form
\begin{equation}
s_{k+1}
=
s_k-a_k s_k^{m+1}+R_k(s_k),
\qquad
m\ge1,
\label{eq:general-parabolic-recurrence}
\end{equation}
where \(a_k\ge0\).

Set
\begin{equation}
A_n
:=
\sum_{k=0}^{n-1}a_k.
\label{eq:An-definition}
\end{equation}

\begin{theorem}[Weighted parabolic comparison]
	\label{thm:weighted-parabolic-comparison}
	Assume that
	\begin{equation}
	0\le a_k\le a_{\max},
	\qquad
	A_n\longrightarrow\infty,
	\label{eq:ak-divergent}
	\end{equation}
	and that, for \(0<s\le r\),
	\begin{equation}
	|R_k(s)|
	\le
	\eta_k a_k s^{m+1},
	\qquad
	\eta_k\longrightarrow0.
	\label{eq:relative-parabolic-remainder}
	\end{equation}
	
	Then, for \(s_0>0\) sufficiently small, the solution of
	\eqref{eq:general-parabolic-recurrence} remains positive,
	\[
	s_k\longrightarrow0,
	\]
	and
	\begin{equation}
	\boxed{
		s_n
		\sim
		\bigl(mA_n\bigr)^{-1/m}.
	}
	\label{eq:master-parabolic-rate}
	\end{equation}
\end{theorem}

\begin{proof}
	After discarding finitely many indices, we may assume
	\[
	\eta_k\le\frac12.
	\]
	Choose \(r>0\) sufficiently small that
	\[
	\frac32 a_{\max}r^m\le\frac12.
	\]
	For a sufficiently small initial value, the finitely many omitted
	steps remain inside \((0,r]\).
	
	Write
	\[
	R_k(s_k)
	=
	\rho_k a_k s_k^{m+1},
	\qquad
	|\rho_k|\le\eta_k.
	\]
	Then
	\begin{equation}
	s_{k+1}
	=
	s_k
	\left[
	1-a_k(1-\rho_k)s_k^m
	\right].
	\label{eq:factor-parabolic-step}
	\end{equation}
	The factor in brackets belongs to \((0,1]\), so the orbit is positive
	and eventually nonincreasing.
	
	If
	\[
	s_k\longrightarrow\ell>0,
	\]
	then for all large \(k\),
	\[
	s_k-s_{k+1}
	\ge
	\frac14 a_k\ell^{m+1}.
	\]
	Since \(\sum_k a_k=\infty\), summation gives a contradiction.
	Hence
	\[
	s_k\longrightarrow0.
	\]
	
	Set
	\begin{equation}
	u_k:=s_k^{-m}.
	\label{eq:uk-master}
	\end{equation}
	With
	\[
	x_k
	=
	a_k(1-\rho_k)s_k^m,
	\]
	one has \(x_k\to0\), and therefore
	\[
	(1-x_k)^{-m}
	=
	1+mx_k+O(x_k^2).
	\]
	Consequently,
	\begin{align}
	u_{k+1}-u_k
	&=
	u_k\left[(1-x_k)^{-m}-1\right]
	\nonumber\\
	&=
	m a_k(1-\rho_k)
	+
	O(a_k^2s_k^m).
	\label{eq:uk-master-increment}
	\end{align}
	
	Because \(\rho_k\to0\),
	\[
	\sum_{k<n}a_k\rho_k=o(A_n).
	\]
	Moreover, \(a_k\) is bounded and \(s_k^m\to0\), so
	\[
	\sum_{k<n}a_k^2s_k^m
	=
	o(A_n).
	\]
	Summing \eqref{eq:uk-master-increment} gives
	\[
	u_n
	=
	mA_n+o(A_n).
	\]
	Since \(u_n=s_n^{-m}\), this proves
	\eqref{eq:master-parabolic-rate}.
\end{proof}

The theorem isolates the essential rule:
\[
\boxed{
	s_{k+1}-s_k
	\sim
	-a_ks_k^{m+1}
	\quad\Longrightarrow\quad
	s_k^{-m}
	\sim
	m\sum_{j<k}a_j.
}
\]

% ---------------------------------------------------------------
\subsection{Cubic regimes}
\label{subsec:abstract-cubic-regime}
% ---------------------------------------------------------------

Consider
\begin{equation}
s_{k+1}
=
s_k-\alpha_k s_k^3+R_k(s_k),
\qquad
\alpha_k\ge0.
\label{eq:abstract-cubic-recurrence}
\end{equation}

\begin{corollary}[Cubic parabolic rate]
	\label{cor:cubic-parabolic-rate}
	Assume
	\[
	\sum_{k=0}^{\infty}\alpha_k=\infty,
	\qquad
	\sup_k\alpha_k<\infty,
	\]
	and
	\begin{equation}
	|R_k(s)|
	\le
	\eta_k\alpha_k|s|^3,
	\qquad
	\eta_k\to0.
	\label{eq:cubic-scaled-remainder}
	\end{equation}
	Then
	\begin{equation}
	\boxed{
		s_n
		\sim
		\left(
		2\sum_{k=0}^{n-1}\alpha_k
		\right)^{-1/2}.
	}
	\label{eq:cubic-general-rate}
	\end{equation}
\end{corollary}

In particular, if
\begin{equation}
\alpha_k
\sim
c(k+1)^{-\theta},
\qquad
c>0,
\qquad
0\le\theta<1,
\label{eq:cubic-power-coefficient}
\end{equation}
then
\begin{equation}
\boxed{
	s_k
	\sim
	\left(
	\frac{2c}{1-\theta}
	k^{1-\theta}
	\right)^{-1/2}.
}
\label{eq:cubic-power-rate}
\end{equation}

At the critical exponent
\[
\alpha_k\sim\frac{c}{k+1},
\]
one instead obtains
\begin{equation}
\boxed{
	s_k
	\sim
	(2c\log k)^{-1/2}.
}
\label{eq:cubic-log-rate}
\end{equation}

\begin{remark}
	\label{rem:cubic-remainder-scale}
	The coefficient-scaled hypothesis
	\eqref{eq:cubic-scaled-remainder} is important.  A generic uniform
	remainder
	\[
	R_k(s)=O(s^5)
	\]
	need not be negligible relative to
	\(\alpha_ks^3\) when \(\alpha_k\to0\).
	The relevant ratio is
	\[
	\frac{s_k^5}{\alpha_ks_k^3}
	=
	\frac{s_k^2}{\alpha_k}.
	\]
	Thus the admissible range under an unweighted \(O(s^5)\) remainder
	depends on the decay of \(\alpha_k\).  The scaled formulation above
	avoids this artificial restriction.
\end{remark}
\begin{remark}[Relation with the nested convex geometry]
	The nonsummable cubic regimes considered in this subsection are abstract.
	For the nested convex geometry of Sections~2--4, the cubic coefficients
	$\beta_k$ are automatically absolutely summable by
	Proposition~\ref{prop:cubic-summability}.  In particular, the harmonic
	cubic regime
	\[
	\alpha_k\sim \frac{c}{k}
	\]
	cannot arise from the geometric cubic coefficient under the standing
	nesting and uniform-curvature assumptions.
\end{remark}
% ---------------------------------------------------------------
\subsection{A quintic regime with nonzero limiting coefficient}
\label{subsec:quintic-constant-regime}
% ---------------------------------------------------------------

We next consider
\begin{equation}
s_{k+1}
=
s_k+\beta_k s_k^3+\gamma_k s_k^5+R_k(s_k),
\label{eq:quintic-constant-recurrence}
\end{equation}
with
\begin{equation}
\gamma_k\longrightarrow-\gamma,
\qquad
\gamma>0.
\label{eq:negative-quintic-limit}
\end{equation}

The cubic term must decay sufficiently rapidly for the quintic drift
to determine the asymptotic scale.

\begin{theorem}[Quintic parabolic rate]
	\label{thm:quintic-constant-rate}
	Assume \eqref{eq:negative-quintic-limit} and suppose that, for some
	\(\eta>0\),
	\begin{equation}
	|\beta_k|
	\le
	C(k+1)^{-1/2-\eta},
	\label{eq:beta-quintic-decay}
	\end{equation}
	while
	\begin{equation}
	|R_k(s)|
	\le
	C|s|^7
	\label{eq:quintic-constant-remainder}
	\end{equation}
	for sufficiently small \(s\), uniformly in \(k\).
	
	Then, for \(s_0>0\) sufficiently small,
	\[
	s_k\longrightarrow0
	\]
	and
	\begin{equation}
	\boxed{
		s_k
		\sim
		(4\gamma k)^{-1/4}.
	}
	\label{eq:quintic-constant-rate}
	\end{equation}
\end{theorem}

\begin{proof}
	Set
	\[
	u_k:=s_k^{-4}.
	\]
	For sufficiently small \(s_k>0\), write
	\[
	s_{k+1}
	=
	s_k(1+z_k),
	\]
	where
	\[
	z_k
	=
	\beta_k s_k^2
	+
	\gamma_k s_k^4
	+
	\rho_k,
	\qquad
	|\rho_k|
	\le
	C s_k^6.
	\]
	After reducing the local interval if necessary, we may assume
	\[
	|z_k|\le\frac12.
	\]
	Hence
	\[
	(1+z_k)^{-4}
	=
	1-4z_k+O(z_k^2),
	\]
	uniformly in \(k\).  Since \(u_k=s_k^{-4}\),
	\begin{equation}
	u_{k+1}-u_k
	=
	-4\gamma_k
	-
	4\beta_k u_k^{1/2}
	+
	E_k,
	\label{eq:quintic-u-recurrence}
	\end{equation}
	where
	\begin{equation}
	|E_k|
	\le
	C\left(
	u_k^{-1/2}
	+
	\beta_k^2
	\right)
	\label{eq:quintic-u-error}
	\end{equation}
	as long as \(u_k\ge1\).
	
	Indeed, the contribution of \(\rho_k\) is
	\[
	u_k O(s_k^6)
	=
	O(s_k^2)
	=
	O(u_k^{-1/2}),
	\]
	whereas
	\[
	u_k z_k^2
	=
	O\!\left(
	\beta_k^2
	+
	u_k^{-1}
	+
	|\beta_k|u_k^{-1/2}
	\right),
	\]
	and the last two terms are absorbed by
	\eqref{eq:quintic-u-error}.
	
	Since
	\[
	\gamma_k\longrightarrow-\gamma,
	\qquad
	\gamma>0,
	\]
	choose \(K\) so large that
	\begin{equation}
	-\gamma_k\ge\frac34\gamma,
	\qquad
	k\ge K.
	\label{eq:gamma-negative-eventually}
	\end{equation}
	Also,
	\[
	\sum_{k=0}^{\infty}\beta_k^2<\infty
	\]
	by \eqref{eq:beta-quintic-decay}.
	
	We now close the bootstrap argument explicitly.  Choose $U>1$ so large that
	\[
	CU^{-1/2}\leq \frac{\gamma}{4},
	\]
	and then choose $K$ sufficiently large that
	\[
	-\gamma_k\geq\frac{3\gamma}{4},
	\qquad
	C_1\beta_k^2\leq\frac12,
	\qquad k\geq K.
	\]
	After reducing the local radius if necessary, choose $s_0>0$ so small that
	the finitely many iterates up to time $K$ remain positive and inside the
	local interval, and so that
	\[
	u_K=s_K^{-4}
	\]
	is arbitrarily large.
	
	Let $N>K$ be the first index, if it exists, for which the bootstrap
	conditions fail:
	\[
	0<s_k\leq r_*,
	\qquad
	u_k\geq U,
	\qquad
	K\leq k<N.
	\]
	Here $r_*>0$ is chosen so that the expansion leading to
	\eqref{eq:quintic-u-recurrence} is valid and $|z_k|\leq 1/2$.
	
	On the interval $K\leq k<N$, Young's inequality gives
	\[
	4|\beta_k|u_k^{1/2}
	\leq
	\frac{\gamma}{2}
	+C_\gamma\beta_k^2u_k,
	\]
	and hence
	\[
	u_{k+1}
	\leq
	(1+C_1\beta_k^2)u_k+C_2,
	\]
	while
	\[
	u_{k+1}
	\geq
	(1-C_1\beta_k^2)u_k+\gamma.
	\]
	Since
	\[
	\sum_{k=K}^{\infty}\beta_k^2<\infty,
	\]
	the products
	\[
	P_{m,n}^{\pm}
	:=
	\prod_{j=m}^{n-1}(1\pm C_1\beta_j^2)
	\]
	satisfy
	\[
	0<p_*
	\leq
	P_{m,n}^{-}
	\leq
	P_{m,n}^{+}
	\leq p^*<\infty
	\]
	for all $n>m\geq K$, after increasing $K$ if necessary.
	
	Iteration of the lower inequality therefore gives
	\[
	u_n
	\geq
	P_{K,n}^{-}u_K
	+
	\gamma
	\sum_{j=K}^{n-1}P_{j+1,n}^{-}
	\geq
	p_*u_K+p_*\gamma(n-K).
	\]
	Choose $s_0$ so small that
	\[
	p_*u_K>U.
	\]
	Then
	\[
	u_n>U
	\]
	for every $K\leq n\leq N$.
	
	Moreover, $u_n\geq U$ implies
	\[
	s_n=u_n^{-1/4}\leq U^{-1/4},
	\]
	which lies inside the chosen local interval, and $|z_n|\leq1/2$ gives
	\[
	s_{n+1}=s_n(1+z_n)>0.
	\]
	Thus none of the bootstrap conditions can fail at $N$, a contradiction.
	Hence no finite first-exit index exists and the estimates hold for every
	$k\geq K$.
	
	Combining the upper and lower estimates gives
	\begin{equation}
	c(1+k)
	\le
	u_k
	\le
	C(1+k)
	\label{eq:quintic-u-two-sided}
	\end{equation}
	for all sufficiently large \(k\).  In particular,
	\[
	u_k\longrightarrow\infty,
	\qquad
	s_k\longrightarrow0.
	\]
	
	Using \eqref{eq:beta-quintic-decay} and
	\eqref{eq:quintic-u-two-sided},
	\[
	|\beta_k|u_k^{1/2}
	=
	O\!\left(
	k^{-1/2-\eta}k^{1/2}
	\right)
	=
	O(k^{-\eta})
	=
	o(1).
	\]
	Moreover,
	\[
	u_k^{-1/2}\to0,
	\qquad
	\beta_k^2\to0.
	\]
	Equation~\eqref{eq:quintic-u-recurrence} therefore reduces to
	\[
	u_{k+1}-u_k
	=
	4\gamma+o(1).
	\]
	Summing from \(K\) to \(n-1\) gives
	\[
	u_n
	=
	4\gamma n+o(n).
	\]
	Since \(u_n=s_n^{-4}\),
	\[
	s_n
	\sim
	(4\gamma n)^{-1/4}.
	\]
\end{proof}

The uniform \(O(s^7)\) remainder is harmless in this regime because
the quintic coefficient remains bounded away from zero.  This changes
in the critical regime considered next.

% ---------------------------------------------------------------
\subsection{Critical quintic logarithmic regime}
\label{subsec:quintic-log-regime}
% ---------------------------------------------------------------

Consider
\begin{equation}
s_{k+1}
=
s_k
+
\beta_k s_k^3
-
\frac{c}{k+1}s_k^5
+
R_k(s_k),
\qquad
c>0.
\label{eq:critical-quintic-recurrence}
\end{equation}

Here a generic uniform \(O(s^7)\) remainder is not sufficiently
small: at the expected scale
\[
s_k\asymp(\log k)^{-1/4},
\]
one has
\[
\frac{s_k^7}{(k+1)^{-1}s_k^5}
=
(k+1)s_k^2
\asymp
\frac{k}{\sqrt{\log k}},
\]
which diverges.

The remainder must therefore carry the same non-autonomous weight as
the quintic drift.

\begin{theorem}[Critical quintic logarithmic law]
	\label{thm:critical-quintic-log}
	Assume that, for some \(\eta>0\),
	\begin{equation}
	|\beta_k|
	\le
	C(k+1)^{-1-\eta},
	\label{eq:beta-critical-quintic}
	\end{equation}
	and that
	\begin{equation}
	|R_k(s)|
	\le
	\frac{C}{k+1}|s|^7
	\label{eq:weighted-seventh-remainder}
	\end{equation}
	for sufficiently small \(s\), uniformly in \(k\).
	
	Then, for \(s_0>0\) sufficiently small,
	\[
	s_k\longrightarrow0
	\]
	and
	\begin{equation}
	\boxed{
		s_k
		\sim
		(4c\log k)^{-1/4}.
	}
	\label{eq:critical-quintic-log-rate}
	\end{equation}
\end{theorem}

\begin{proof}
	Set
	\[
	u_k:=s_k^{-4}.
	\]
	Write
	\[
	s_{k+1}
	=
	s_k(1+z_k),
	\]
	where
	\[
	z_k
	=
	\beta_k s_k^2
	-
	\frac{c}{k+1}s_k^4
	+
	\rho_k,
	\qquad
	|\rho_k|
	\le
	\frac{C}{k+1}s_k^6.
	\]
	For sufficiently small \(s_k\),
	\[
	|z_k|\le\frac12,
	\]
	and therefore
	\[
	(1+z_k)^{-4}
	=
	1-4z_k+O(z_k^2).
	\]
	Since \(u_k=s_k^{-4}\), we obtain
	\begin{equation}
	u_{k+1}-u_k
	=
	\frac{4c}{k+1}
	-
	4\beta_k u_k^{1/2}
	+
	E_k,
	\label{eq:critical-quintic-u}
	\end{equation}
	with
	\begin{equation}
	|E_k|
	\le
	C\left[
	\frac{u_k^{-1/2}}{k+1}
	+
	\beta_k^2
	\right]
	\label{eq:critical-quintic-error}
	\end{equation}
	whenever \(u_k\ge1\).
	
	Indeed,
	\[
	u_k\rho_k
	=
	O\!\left(
	\frac{u_k^{-1/2}}{k+1}
	\right),
	\]
	while the quadratic Taylor remainder contributes only terms bounded
	by
	\[
	C\left(
	\beta_k^2
	+
	\frac{u_k^{-1}}{(k+1)^2}
	+
	\frac{|\beta_k|u_k^{-1/2}}{k+1}
	\right),
	\]
	which are absorbed into
	\eqref{eq:critical-quintic-error}.
	
	Choose
	\[
	q_k:=(k+1)^{-1-\eta}.
	\]
	Then
	\[
	\sum_k q_k<\infty.
	\]
	By Young's inequality,
	\begin{equation}
	4|\beta_k|u_k^{1/2}
	\le
	q_k u_k
	+
	C q_k^{-1}\beta_k^2.
	\label{eq:critical-quintic-young}
	\end{equation}
	Using
	\[
	|\beta_k|
	\le
	C(k+1)^{-1-\eta},
	\]
	we have
	\[
	q_k^{-1}\beta_k^2
	=
	O((k+1)^{-1-\eta}),
	\]
	and therefore
	\begin{equation}
	\sum_k q_k^{-1}\beta_k^2<\infty.
	\label{eq:critical-quintic-young-summable}
	\end{equation}
	
	We now justify the bootstrap globally.  Choose $U>1$ so large that
	\[
	CU^{-1/2}\leq\frac{c}{2}.
	\]
	Increase $K$ if necessary so that
	\[
	Cq_k\leq\frac12,
	\qquad k\geq K,
	\]
	where
	\[
	q_k=(k+1)^{-1-\eta}.
	\]
	Choose $s_0>0$ sufficiently small that the finitely many iterates up to
	time $K$ remain positive and inside the local interval and that $u_K$ is
	arbitrarily large.
	
	Let $N>K$ be the first index, if one exists, for which
	\[
	0<s_k\leq r_*,
	\qquad
	u_k\geq U,
	\qquad
	K\leq k<N,
	\]
	fails.  On this maximal bootstrap interval, the preceding estimates give
	\[
	u_{k+1}
	\leq
	(1+Cq_k)u_k+\frac{C}{k+1}+d_k
	\]
	and
	\[
	u_{k+1}
	\geq
	(1-Cq_k)u_k+\frac{c_0}{k+1}-d_k,
	\]
	where
	\[
	\sum_{k=K}^{\infty}q_k<\infty,
	\qquad
	\sum_{k=K}^{\infty}d_k<\infty.
	\]
	
	Set
	\[
	P_{m,n}^{-}
	:=
	\prod_{j=m}^{n-1}(1-Cq_j),
	\qquad
	P_{m,n}^{+}
	:=
	\prod_{j=m}^{n-1}(1+Cq_j).
	\]
	Since $\sum q_k<\infty$, there exist constants
	\[
	0<p_*\leq p^*<\infty
	\]
	such that
	\[
	p_*\leq P_{m,n}^{-}\leq P_{m,n}^{+}\leq p^*
	\]
	for all $n>m\geq K$.
	
	Iteration of the lower inequality gives
	\[
	\begin{aligned}
	u_n
	&\geq
	P_{K,n}^{-}u_K
	+
	\sum_{j=K}^{n-1}
	P_{j+1,n}^{-}
	\left(
	\frac{c_0}{j+1}-d_j
	\right)
	\\
	&\geq
	p_*u_K
	+
	p_*c_0
	\sum_{j=K}^{n-1}\frac1{j+1}
	-
	\sum_{j=K}^{\infty}d_j .
	\end{aligned}
	\]
	Choose $s_0$ so small that
	\[
	p_*u_K-\sum_{j=K}^{\infty}d_j>U.
	\]
	Then
	\[
	u_n>U
	\]
	throughout the maximal bootstrap interval.
	
	As before, $u_n\geq U$ keeps $s_n$ inside the local region, and
	$|z_n|\leq1/2$ implies
	\[
	s_{n+1}=s_n(1+z_n)>0.
	\]
	Thus a finite first exit is impossible, and the estimates hold for every
	$n\geq K$.
	
	The upper inequality gives
	\[
	u_n=O(\log(n+2)),
	\]
	whereas the preceding lower estimate gives
	\[
	u_n
	\geq
	c_1u_K+c_2\log n-C_2.
	\]
	In particular,
	\[
	u_n\to\infty.
	\]
	
We can now return to the precise recurrence
\eqref{eq:critical-quintic-u}. Since
\[
u_k=O(\log(k+2)),
\]
we obtain
\[
|\beta_k|u_k^{1/2}
=
O\!\left(
(k+1)^{-1-\eta}
(\log(k+2))^{1/2}
\right)
=
o\!\left(\frac1{k+1}\right).
\]
Furthermore, since \(u_k\to\infty\),
\[
\frac{u_k^{-1/2}}{k+1}
=
o\!\left(\frac1{k+1}\right).
\]
Moreover,
\[
\beta_k^2
	=
	O((k+1)^{-2-2\eta})
	=
	o\!\left(\frac1{k+1}\right).
	\]
	Thus
	\begin{equation}
	u_{k+1}-u_k
	=
	\frac{4c}{k+1}
	+
	o\!\left(\frac1{k+1}\right).
	\label{eq:critical-quintic-final-increment}
	\end{equation}
	
	Summing,
	\[
	u_n
	=
	4c
	\sum_{k<n}\frac1{k+1}
	+
	o(\log n)
	=
	4c\log n+o(\log n).
	\]
	Therefore
	\[
	s_n
	=
	u_n^{-1/4}
	\sim
	(4c\log n)^{-1/4}.
	\]
\end{proof}

\begin{remark}[More general cubic perturbations]
	\label{rem:weighted-beta-condition}
	The pointwise assumption
	\[
	|\beta_k|\le C(k+1)^{-1-\eta}
	\]
	is sufficient but not optimal.  Once
	\[
	u_k=O(\log k)
	\]
	has been established, the transformed cubic contribution is
	\[
	\beta_k u_k^{1/2}.
	\]
	Thus the asymptotic conclusion remains valid whenever
	\[
	\sum_{k<n}
	|\beta_k|(\log(k+2))^{1/2}
	=
	o(\log n),
	\]
	provided the corresponding a priori bounds and positivity of the
	orbit are retained.
\end{remark}

% ---------------------------------------------------------------
\subsection{Critical quartic logarithmic regime}
\label{subsec:quartic-log-regime}
% ---------------------------------------------------------------

Without local reflection symmetry, the first term after the cubic
contribution may be quartic.  The abstract critical model is
\begin{equation}
s_{k+1}
=
s_k
+
\beta_k s_k^3
-
\frac{c}{k+1}s_k^4
+
R_k(s_k),
\qquad
c>0.
\label{eq:critical-quartic-recurrence}
\end{equation}

As in the quintic case, an unweighted \(O(s^5)\) remainder is too
large at the logarithmic scale.  We therefore impose a weighted
remainder.

\begin{theorem}[Critical quartic logarithmic law]
	\label{thm:critical-quartic-log}
	Assume that, for some \(\eta>0\),
	\begin{equation}
	|\beta_k|
	\le
	C(k+1)^{-1-\eta},
	\label{eq:beta-critical-quartic}
	\end{equation}
	and
	\begin{equation}
	|R_k(s)|
	\le
	\frac{C}{k+1}|s|^5
	\label{eq:weighted-fifth-remainder}
	\end{equation}
	for small \(s\), uniformly in \(k\).
	
	Then, for sufficiently small \(s_0>0\),
	\[
	s_k\longrightarrow0
	\]
	and
	\begin{equation}
	\boxed{
		s_k
		\sim
		(3c\log k)^{-1/3}.
	}
	\label{eq:critical-quartic-log-rate}
	\end{equation}
\end{theorem}

\begin{proof}
	Set
	\[
	u_k:=s_k^{-3}.
	\]
	Write
	\[
	s_{k+1}
	=
	s_k(1+z_k),
	\]
	with
	\[
	z_k
	=
	\beta_k s_k^2
	-
	\frac{c}{k+1}s_k^3
	+
	\rho_k,
	\qquad
	|\rho_k|
	\le
	\frac{C}{k+1}s_k^4.
	\]
	For sufficiently small \(s_k\),
	\[
	|z_k|\le\frac12,
	\]
	so that
	\[
	(1+z_k)^{-3}
	=
	1-3z_k+O(z_k^2).
	\]
	Since \(u_k=s_k^{-3}\), this gives
	\begin{equation}
	u_{k+1}-u_k
	=
	\frac{3c}{k+1}
	-
	3\beta_k u_k^{1/3}
	+
	E_k,
	\label{eq:critical-quartic-u}
	\end{equation}
	where
	\begin{equation}
	|E_k|
	\le
	C\left[
	\frac{u_k^{-1/3}}{k+1}
	+
	\beta_k^2
	\right]
	\label{eq:critical-quartic-error}
	\end{equation}
	whenever \(u_k\ge1\).
	
	Indeed, the weighted fifth-order remainder contributes
	\[
	u_k O\!\left(
	\frac{s_k^4}{k+1}
	\right)
	=
	O\!\left(
	\frac{u_k^{-1/3}}{k+1}
	\right),
	\]
	and the quadratic Taylor terms are of still lower order.
	
	Set
	\[
	q_k:=(k+1)^{-1-\eta}.
	\]
	Then
	\[
	\sum_kq_k<\infty.
	\]
	Young's inequality with conjugate exponents \(3\) and \(3/2\)
	gives
	\begin{equation}
	3|\beta_k|u_k^{1/3}
	\le
	Cq_k u_k
	+
	Cq_k^{-1/2}|\beta_k|^{3/2}.
	\label{eq:quartic-young}
	\end{equation}
	Because
	\[
	|\beta_k|
	\le
	C(k+1)^{-1-\eta},
	\]
	we have
	\[
	q_k^{-1/2}|\beta_k|^{3/2}
	=
	O((k+1)^{-1-\eta}),
	\]
	which is summable.
	
	We again formulate the argument on a maximal bootstrap interval.  Choose
	$U>1$ so large that
	\[
	CU^{-1/3}\leq\frac{c}{2},
	\]
	and increase $K$ so that
	\[
	Cq_k\leq\frac12,
	\qquad
	q_k=(k+1)^{-1-\eta},
	\qquad
	k\geq K.
	\]
	Choose $s_0>0$ sufficiently small that the finitely many iterates up to
	time $K$ remain positive and in the local region, with $u_K=s_K^{-3}$
	arbitrarily large.
	
	Let $N>K$ be the first index, if one exists, for which
	\[
	0<s_k\leq r_*,
	\qquad
	u_k\geq U,
	\qquad
	K\leq k<N,
	\]
	fails.  On this interval,
	\[
	u_{k+1}
	\leq
	(1+Cq_k)u_k+\frac{C}{k+1}+d_k
	\]
	and
	\[
	u_{k+1}
	\geq
	(1-Cq_k)u_k+\frac{c_0}{k+1}-d_k,
	\]
	with
	\[
	\sum q_k<\infty,
	\qquad
	\sum d_k<\infty.
	\]
	
	Using the products
	\[
	P_{m,n}^{\pm}
	=
	\prod_{j=m}^{n-1}(1\pm Cq_j),
	\]
	there exist constants $0<p_*\leq p^*<\infty$ such that
	\[
	p_*\leq P_{m,n}^{-}\leq P_{m,n}^{+}\leq p^*
	\]
	for all $n>m\geq K$.  Consequently,
	\[
	\begin{aligned}
	u_n
	&\geq
	p_*u_K
	+
	p_*c_0\sum_{j=K}^{n-1}\frac1{j+1}
	-
	\sum_{j=K}^{\infty}d_j .
	\end{aligned}
	\]
	Taking $s_0$ sufficiently small ensures
	\[
	p_*u_K-\sum_{j=K}^{\infty}d_j>U.
	\]
	Hence $u_n>U$ throughout the bootstrap interval.
	
	Since $u_n\geq U$ keeps
	\[
	s_n=u_n^{-1/3}
	\]
	inside the local neighbourhood and $|z_n|\leq1/2$ preserves positivity,
	the first-exit index cannot be finite.  Therefore the estimates hold for
	all $n\geq K$.
	
	Iteration of the upper inequality yields
	\[
	u_n=O(\log(n+2)),
	\]
	while the lower estimate gives
	\[
	u_n\geq c_1u_K+c_2\log n-C_2,
	\]
	so in particular
	\[
	u_n\to\infty.
	\]
	
	Since
	\[
	u_k=O(\log(k+2)),
	\]
	we have
	\[
	|\beta_k|u_k^{1/3}
	=
	O\!\left(
	(k+1)^{-1-\eta}
	(\log(k+2))^{1/3}
	\right)
	=
	o\!\left(\frac1{k+1}\right).
	\]
	Also, since \(u_k\to\infty\),
	\[
	\frac{u_k^{-1/3}}{k+1}
	=
	o\!\left(\frac1{k+1}\right).
	\]
	Moreover,
	\[
	\beta_k^2
	=
	o\!\left(\frac1{k+1}\right).
	\]
	Thus
	\[
	u_{k+1}-u_k
	=
	\frac{3c}{k+1}
	+
	o\!\left(\frac1{k+1}\right).
	\]
	Summation gives
	\[
	u_n
	=
	3c\log n+o(\log n).
	\]
	Since \(u_n=s_n^{-3}\),
	\[
	s_n
	\sim
	(3c\log n)^{-1/3}.
	\]
\end{proof}

More generally, it is enough that
\[
\sum_{k<n}
|\beta_k|(\log(k+2))^{1/3}
=
o(\log n).
\]

% ---------------------------------------------------------------
\subsection{Relation to the geometric model}
\label{subsec:abstract-versus-geometric}
% ---------------------------------------------------------------

The preceding logarithmic regimes must be distinguished from the
summable geometric regime of Section~\ref{sec:summable-geometry}.

\begin{proposition}[Separation of geometric and critical regimes]
	\label{prop:separation-geometric-critical}
	Assume the summable geometric variation hypothesis
	\[
	\sum_{k=0}^{\infty}
	\|f_{k+1}-f_k\|_{C^2}
	<\infty.
	\]
	Then no sufficiently small nonzero geometric orbit can satisfy any
	of the asymptotic laws
	\[
	s_k\to0
	\]
	of
	Corollary~\ref{cor:cubic-parabolic-rate},
	Theorem~\ref{thm:quintic-constant-rate},
	Theorem~\ref{thm:critical-quintic-log}, or
	Theorem~\ref{thm:critical-quartic-log}.
\end{proposition}

\begin{proof}
	By Theorem~\ref{thm:geometric-summable-convergence}, every
	sufficiently small nonzero geometric orbit satisfies
	\[
	s_k\longrightarrow s_\infty\neq0.
	\]
	This is incompatible with any asymptotic regime in which
	\(s_k\to0\).
\end{proof}

Thus the critical recurrences describe a genuinely different
non-autonomous regime.  If one of the logarithmic laws is to be
realized by nested convex boundaries, then the corresponding family
must fall outside the summable-transition-defect regime of
Section~\ref{sec:summable-geometry}.

In particular, the formal occurrence of a coefficient such as
\[
\gamma_k\sim-\frac{c}{k}
\]
or
\[
\delta_k\sim-\frac{c}{k}
\]
is not by itself sufficient.  Geometric realizability must also
produce the weighted control of the complete higher-order remainder
required by
\eqref{eq:weighted-seventh-remainder} or
\eqref{eq:weighted-fifth-remainder}.

Determining whether such critical abstract regimes can actually arise
from nested convex geometry is therefore a separate realizability
problem.

% ================================================================
\section{Explicit Local Geometric Models and Numerical Checks}
\label{sec:examples-numerics}
% ================================================================

We illustrate the preceding theory using an explicit local parabolic
boundary-germ family and two synthetic critical recurrences.  The geometric
and abstract examples are kept separate.  The former provides an exact local
model in the summable-variation regime of Section~4, whereas the latter
illustrate the nonsummable critical mechanisms of Section~5 without asserting
their geometric realizability.

% ---------------------------------------------------------------
\subsection{A nested parabolic boundary-germ family}
\label{subsec:nested-parabolic-family}
% ---------------------------------------------------------------

Consider the sequence of local convex boundary germs
\begin{equation}\label{eq:explicit-kappa-family}
f_k(x)=-\frac{\kappa_k}{2}x^2,
\qquad
\kappa_k
=
1-\frac1{k+2}
=
\frac{k+1}{k+2}.
\end{equation}

These germs share the origin and the same tangent line there, and satisfy
\[
f_{k+1}(x)\leq f_k(x)
\]
on the common local chart.  Thus they form a nested local model for the
transition geometry developed in Sections~2--4.
Then
\[
\frac12=\kappa_0<\kappa_1<\cdots<1
\]
and
\begin{equation}
\Delta\kappa_k
:=
\kappa_{k+1}-\kappa_k
=
\frac{1}{(k+2)(k+3)}.
\label{eq:delta-kappa-family}
\end{equation}

Thus
\begin{equation}
\sum_{k=0}^{\infty}\Delta\kappa_k
=
\frac12.
\label{eq:delta-kappa-sum}
\end{equation}

On a fixed interval \(I_{\rho_0}\), with
\(0<\rho_0\le1\),
\[
f_{k+1}(x)-f_k(x)
=
-\frac{\Delta\kappa_k}{2}x^2.
\]
Consequently,
\begin{equation}
\varepsilon_k
=
\|f_{k+1}-f_k\|_{C^2(I_{\rho_0})}
\le
C_{\rho_0}\Delta\kappa_k,
\label{eq:parabolic-epsilon}
\end{equation}
and therefore
\begin{equation}
\boxed{
	\sum_{k=0}^{\infty}\varepsilon_k<\infty.
}
\label{eq:parabolic-summable-variation}
\end{equation}

Hence the local geometric variation is summable. The family therefore
satisfies the local summability condition underlying
Theorem~\ref{thm:geometric-summable-convergence}.
Its exact transition dynamics provides a concrete model of the finite-drift
regime established there.

% ============================================================
% ================================================================
\begin{figure}[H]
	\centering
	
	\begin{tikzpicture}[
	scale=1.18,
	>=Latex,
	line cap=round,
	line join=round
	]
	
	% ------------------------------------------------------------
	% Parameters
	% ------------------------------------------------------------
	\def\a{0.70}      % kappa_k
	\def\b{1.10}      % kappa_{k+1}
	\def\xc{1.25}     % chosen x-coordinate on \partial C_k
	
	% ------------------------------------------------------------
	% Curves:
	% f_k(x)     = -(a/2)x^2
	% f_{k+1}(x) = -(b/2)x^2
	% ------------------------------------------------------------
	
	\draw[thick,domain=-2.1:2.1,smooth,samples=150]
	plot (\x,{-(\a/2)*\x*\x});
	
	\draw[thick,dashed,domain=-1.85:1.85,smooth,samples=150]
	plot (\x,{-(\b/2)*\x*\x});
	
	% Boundary labels
	\node[right] at (1.80,-1.12) {$\partial C_k$};
	\node[right] at (1.42,-1.65) {$\partial C_{k+1}$};
	
	% ------------------------------------------------------------
	% Common tangency point
	% ------------------------------------------------------------
	\coordinate (p) at (0,0);
	\fill (p) circle (1.6pt);
	\node[above left] at (p) {$p$};
	
	% Common tangent
	\draw[densely dotted] (-2.25,0) -- (2.25,0);
	\node[above right] at (1.40,0.03) {common tangent};
	
	% ------------------------------------------------------------
	% Starting point c_k(x)
	% ------------------------------------------------------------
	\pgfmathsetmacro{\yc}{-(\a/2)*\xc*\xc}
	
	\coordinate (ck) at (\xc,\yc);
	
	\fill (ck) circle (1.6pt);
	\node[above right] at (ck) {$c_k(x)$};
	
	% ------------------------------------------------------------
	% Exact first-hit parameter
	%
	% lambda =
	% ((b-a)x^2) /
	% (1+abx^2 + sqrt(1+2abx^2+a^3 b x^4))
	% ------------------------------------------------------------
	
	\pgfmathsetmacro{\lamhit}{
		((\b-\a)*\xc*\xc)/
		(
		1+\a*\b*\xc*\xc
		+sqrt(
		1
		+2*\a*\b*\xc*\xc
		+\a*\a*\a*\b*\xc*\xc*\xc*\xc
		)
		)
	}
	
	% Exact first-hit point
	\pgfmathsetmacro{\Xhit}{\xc*(1-\a*\lamhit)}
	\pgfmathsetmacro{\Yhit}{-(\b/2)*\Xhit*\Xhit}
	
	\coordinate (gk) at (\Xhit,\Yhit);
	
	\fill (gk) circle (1.6pt);
	\node[left] at (gk) {$\mathcal G_k(c_k(x))$};
	
	% ------------------------------------------------------------
	% Inward normal segment
	% ------------------------------------------------------------
	\draw[->,thick]
	(ck)
	--
	node[midway,right] {$-\nu_k(x)$}
	(gk);
	
	% ------------------------------------------------------------
	% Nesting label
	% ------------------------------------------------------------
	\node at (-1.35,-0.95) {$C_{k+1}\subset C_k$};
	
	\end{tikzpicture}
	
	\caption{Local geometric transition for the nested parabolic boundary-germ
		family. The boundaries $\partial C_k$ and $\partial C_{k+1}$ share the common
		tangency point $p$. Starting from $c_k(x)\in\partial C_k$, the transition
		$\mathcal G_k(c_k(x))\in\partial C_{k+1}$ is the first intersection of the inward
		normal ray with the next boundary. The displayed intersection is computed
		from the exact quadratic-germ formula of Section~6.2.}
	\label{fig:nested-parabolic-transition}
	
\end{figure}
% ================================================================
% ============================================================
% ---------------------------------------------------------------
\subsection{Exact transition for quadratic boundary germs}
\label{subsec:exact-parabolic-transition}
% ---------------------------------------------------------------

For quadratic graphs the local normal intersection can be written
explicitly.

Let
\[
a:=\kappa_k,
\qquad
b:=\kappa_{k+1},
\qquad
b\ge a>0.
\]
Starting from
\[
c_a(x)=
\left(
x,-\frac a2x^2
\right),
\]
the unnormalised inward normal direction is
\[
(-ax,-1).
\]
Write the first intersection with the next parabola as
\[
c_b(X)
=
c_a(x)-\lambda(ax,1),
\qquad
\lambda\ge0.
\]
Then
\begin{equation}
X=x(1-a\lambda)
\label{eq:parabolic-X-exact}
\end{equation}
and \(\lambda\) satisfies
\begin{equation}
ba^2x^2\lambda^2
-
2(1+abx^2)\lambda
+
(b-a)x^2
=
0.
\label{eq:parabolic-lambda-quadratic}
\end{equation}

The smaller nonnegative root, corresponding to the first
intersection, is
\begin{equation}
\boxed{
	\lambda
	=
	\frac{
		(b-a)x^2
	}{
		1+abx^2+
		\sqrt{
			1+2abx^2+a^3bx^4
		}
	}.
}
\label{eq:parabolic-lambda-exact}
\end{equation}

This form is also numerically preferable to the equivalent
difference-of-square-roots expression, since it avoids cancellation
when \(b-a\) is small.

The signed arclength on the parabola of curvature parameter \(a\) is
\begin{equation}
s_a(x)
=
\int_0^x\sqrt{1+a^2u^2}\,du
=
\frac12
\left[
x\sqrt{1+a^2x^2}
+
\frac{\operatorname{arsinh}(ax)}{a}
\right].
\label{eq:exact-parabolic-arclength}
\end{equation}

Hence the scalar transition is determined exactly by
\begin{equation}
G_k(s_a(x))
=
s_b\!\left(
x(1-a\lambda)
\right),
\label{eq:exact-parabolic-G}
\end{equation}
with \(\lambda\) given by
\eqref{eq:parabolic-lambda-exact}.

Expanding this exact expression reproduces the normal form of
Section~\ref{sec:local-parabolic-normal-form}:
\begin{equation}
G_k(s)
=
s+\beta_ks^3+\gamma_ks^5+O(s^7),
\label{eq:parabolic-family-normal-form}
\end{equation}
where
\begin{equation}
\beta_k
=
-\frac{
	k^2+4k+2
}{
	6(k+2)^2(k+3)^2
}
<0
\label{eq:explicit-beta-family}
\end{equation}
and
\begin{equation}
\gamma_k
=
\frac{
	28k^6+336k^5+1568k^4+3584k^3
	+4130k^2+2184k+375
}{
	120(k+2)^4(k+3)^4
}.
\label{eq:explicit-gamma-family}
\end{equation}

In particular,
\begin{equation}
\beta_k=O(k^{-2}),
\qquad
\gamma_k=O(k^{-2}),
\label{eq:beta-gamma-decay-family}
\end{equation}
so both coefficient sequences are summable.

The cubic coefficients admit the exact sum
\begin{equation}
\boxed{
	\sum_{k=0}^{\infty}\beta_k
	=
	\frac{\pi^2}{12}-\frac78
	=
	-0.052532966575\ldots
}
\label{eq:exact-beta-sum}
\end{equation}
This coefficient-level summability is consistent with, but weaker
than, the full geometric transition-defect estimate of Section~4.

% ---------------------------------------------------------------
\subsection{Finite drift in the exact geometric model}
\label{subsec:exact-geometric-drift}
% ---------------------------------------------------------------

We now iterate the exact transition
\eqref{eq:exact-parabolic-G}, rather than a truncated Taylor
recurrence.

Take
\begin{equation}
s_0=10^{-3}.
\label{eq:numerical-geometric-s0}
\end{equation}
At each step, \(x_k\) is obtained from
\[
s_{\kappa_k}(x_k)=s_k,
\]
the normal intersection is computed from
\eqref{eq:parabolic-lambda-exact}, and
\[
s_{k+1}
=
s_{\kappa_{k+1}}(x_{k+1}).
\]

Representative values are shown in
Table~\ref{tab:exact-geometric-drift}.

\begin{table}[ht]
	\centering
	\begin{tabular}{rcc}
		\hline
		\(k\)
		&
		\(s_k\)
		&
		\(s_0-s_k\)
		\\
		\hline
		\(0\)
		&
		\(1.000000000000\times10^{-3}\)
		&
		\(0\)
		\\
		\(10\)
		&
		\(9.999999607291\times10^{-4}\)
		&
		\(3.92709\times10^{-11}\)
		\\
		\(10^4\)
		&
		\(9.999999474837\times10^{-4}\)
		&
		\(5.25163\times10^{-11}\)
		\\
		\(10^6\)
		&
		\(9.999999474659\times10^{-4}\)
		&
		\(5.25341\times10^{-11}\)
		\\
		\hline
	\end{tabular}
	\caption{Iteration of the exact parabolic transition for
		\(\kappa_k=1-(k+2)^{-1}\) and \(s_0=10^{-3}\).}
	\label{tab:exact-geometric-drift}
\end{table}

The orbit stabilizes rapidly at the scale shown:
\begin{equation}
s_\infty
\approx
9.99999947466\times10^{-4}.
\label{eq:numerical-geometric-limit}
\end{equation}
Thus
\begin{equation}
s_\infty\neq0,
\end{equation}
in agreement with
Theorem~\ref{thm:geometric-summable-convergence}.

The observed total drift is
\begin{equation}
s_0-s_\infty
\approx
5.2534\times10^{-11}.
\label{eq:numerical-total-drift}
\end{equation}

The cubic approximation alone predicts
\begin{equation}
-s_0^3
\sum_{k=0}^{\infty}\beta_k
=
5.2532966576\times10^{-11},
\label{eq:cubic-drift-prediction}
\end{equation}
which already accounts for the drift to about four significant
digits beyond the displayed scale.  The remaining difference comes
from the quintic and higher-order terms of the exact transition.

This example illustrates the central geometric conclusion of
Section~4: the infinitely many weak parabolic transitions produce a
finite accumulated displacement rather than attraction to the common
tangency point.

% ---------------------------------------------------------------
\subsection{Synthetic critical quintic dynamics}
\label{subsec:numerical-quintic-critical}
% ---------------------------------------------------------------

To illustrate
Theorem~\ref{thm:critical-quintic-log}, consider the purely abstract
recurrence
\begin{equation}
s_{k+1}
=
s_k
-
\frac{1}{k+1}s_k^5,
\qquad
s_0=\frac12.
\label{eq:synthetic-quintic-recurrence}
\end{equation}
Here
\[
\beta_k=0,
\qquad
R_k=0,
\qquad
c=1,
\]
so all hypotheses of the theorem are satisfied exactly.  Therefore
\begin{equation}
s_k
\sim
(4\log k)^{-1/4}.
\label{eq:synthetic-quintic-rate}
\end{equation}

A convenient normalized quantity is
\begin{equation}
Q_k^{(5)}
:=
s_k(4\log k)^{1/4},
\label{eq:normalized-quintic}
\end{equation}
for which
\[
Q_k^{(5)}\longrightarrow1.
\]

\begin{table}[ht]
	\centering
	\begin{tabular}{rcc}
		\hline
		\(k\) & \(s_k\) & \(Q_k^{(5)}\)\\
		\hline
		\(10^2\) & \(0.403456\) & \(0.835840\)\\
		\(10^4\) & \(0.365320\) & \(0.900031\)\\
		\(10^6\) & \(0.340304\) & \(0.927840\)\\
		\hline
	\end{tabular}
	\caption{Critical quintic recurrence
		\eqref{eq:synthetic-quintic-recurrence}.  The convergence of the
		normalized quantity to \(1\) is logarithmically slow.}
	\label{tab:synthetic-quintic}
\end{table}

The slow approach of \(Q_k^{(5)}\) to \(1\) is intrinsic to the
logarithmic regime and should not be confused with a failure of the
asymptotic law.

% ---------------------------------------------------------------
\subsection{Synthetic critical quartic dynamics}
\label{subsec:numerical-quartic-critical}
% ---------------------------------------------------------------

The nonsymmetric critical model of
Theorem~\ref{thm:critical-quartic-log} is illustrated by
\begin{equation}
s_{k+1}
=
s_k
-
\frac{1}{k+1}s_k^4,
\qquad
s_0=\frac12.
\label{eq:synthetic-quartic-recurrence}
\end{equation}
Again
\[
\beta_k=0,
\qquad
R_k=0,
\qquad
c=1,
\]
and hence
\begin{equation}
s_k
\sim
(3\log k)^{-1/3}.
\label{eq:synthetic-quartic-rate}
\end{equation}

Define
\begin{equation}
Q_k^{(4)}
:=
s_k(3\log k)^{1/3}.
\label{eq:normalized-quartic}
\end{equation}
Then
\[
Q_k^{(4)}\longrightarrow1.
\]

\begin{table}[ht]
	\centering
	\begin{tabular}{rcc}
		\hline
		\(k\) & \(s_k\) & \(Q_k^{(4)}\)\\
		\hline
		\(10^2\) & \(0.342961\) & \(0.822937\)\\
		\(10^4\) & \(0.295915\) & \(0.894607\)\\
		\(10^6\) & \(0.267220\) & \(0.924765\)\\
		\hline
	\end{tabular}
	\caption{Critical quartic recurrence
		\eqref{eq:synthetic-quartic-recurrence}.}
	\label{tab:synthetic-quartic}
\end{table}

The quartic and quintic examples exhibit the same structural rule:
if
\[
s_{k+1}-s_k
\sim
-\frac{c}{k}s_k^{m+1},
\]
then
\[
s_k^{-m}
\sim
mc\log k.
\]

% ================================================================
\begin{figure}[H]
	\centering
	
	\begin{tikzpicture}
	
	\begin{axis}[
	width=0.88\linewidth,
	height=7.0cm,
	xmode=log,
	log basis x=10,
	xmin=1e2,
	xmax=1e8,
	ymin=0.20,
	ymax=0.43,
	xlabel={iteration index \(k\)},
	ylabel={leading-order amplitude},
	grid=major,
	legend style={
		at={(0.97,0.97)},
		anchor=north east,
		draw=none,
		fill=none
	},
	tick label style={font=\small},
	label style={font=\small},
	]
	
	% ------------------------------------------------------------
	% Quintic logarithmic law
	% s_k ~ (4 log k)^(-1/4)
	% ------------------------------------------------------------
	\addplot[
	thick,
	domain=1e2:1e8,
	samples=160
	]
	{pow(4*ln(x),-0.25)};
	
	\addlegendentry{
		quintic:
		\(\displaystyle (4\log k)^{-1/4}\)
	}
	
	% ------------------------------------------------------------
	% Quartic logarithmic law
	% s_k ~ (3 log k)^(-1/3)
	% ------------------------------------------------------------
	\addplot[
	thick,
	dashed,
	domain=1e2:1e8,
	samples=160
	]
	{pow(3*ln(x),-1/3)};
	
	\addlegendentry{
		quartic:
		\(\displaystyle (3\log k)^{-1/3}\)
	}
	
	% ------------------------------------------------------------
	% Reference points for the quintic synthetic recurrence
	% ------------------------------------------------------------
	\addplot[
	only marks,
	mark=*,
	mark size=1.5pt
	]
	coordinates {
		(1e2,0.403456)
		(1e4,0.365320)
		(1e6,0.340304)
	};
	
	% ------------------------------------------------------------
	% Reference points for the quartic synthetic recurrence
	% ------------------------------------------------------------
	\addplot[
	only marks,
	mark=square*,
	mark size=1.5pt
	]
	coordinates {
		(1e2,0.342961)
		(1e4,0.295915)
		(1e6,0.267220)
	};
	
	\end{axis}
	
	\end{tikzpicture}
	
	\caption{
		Critical logarithmic decay in the synthetic non-autonomous
		recurrences.  The solid curve shows the quintic law
		\(
		s_k\sim(4\log k)^{-1/4}
		\),
		whereas the dashed curve shows the quartic law
		\(
		s_k\sim(3\log k)^{-1/3}
		\).
		The markers indicate representative numerical iterates of the
		corresponding recurrences from
		Sections~\ref{subsec:numerical-quintic-critical} and
		\ref{subsec:numerical-quartic-critical}.
		Both convergences are logarithmically slow, but the exponent is
		determined by the order of the first effective nonlinear drift.
	}
	\label{fig:critical-logarithmic-regimes}
	
\end{figure}
% ================================================================

% ---------------------------------------------------------------
\subsection{Geometric versus synthetic behavior}
\label{subsec:geometric-versus-synthetic-numerics}
% ---------------------------------------------------------------

The examples above make the distinction between the two parts of the
paper explicit.

For the nested parabolic boundary-germ model,
\[
\sum_k\varepsilon_k<\infty,
\]
the complete transition defect is summable, and
\[
s_k\longrightarrow s_\infty\neq0.
\]

For the synthetic critical recurrences,
\[
\sum_k\frac1{k+1}=\infty,
\]
and the nonsummable parabolic drift produces
\[
s_k\longrightarrow0
\]
with logarithmic rates.

Thus the numerical examples do not provide evidence that the
critical logarithmic regimes are geometrically realizable.  They
only confirm the abstract asymptotic mechanisms of Section~5.
Whether nested convex geometry can generate the weighted critical
defects required there remains a separate realizability question.
% ================================================================
\section{Discussion and Open Problems}
\label{sec:discussion-open-problems}
% ================================================================

The preceding sections separate two qualitatively different forms of
non-autonomous parabolic dynamics.

For nested convex boundaries with summable geometric variation,
\[
\sum_{k=0}^{\infty}
\|f_{k+1}-f_k\|_{C^2}<\infty,
\]
the complete transition defect is summable and every sufficiently
small nonzero orbit satisfies
\[
s_k\longrightarrow s_\infty\neq0.
\]
The common tangency point is therefore Lyapunov stable but not locally
attractive.

By contrast, the critical abstract recurrences of
Section~\ref{sec:abstract-recurrences} are driven by nonsummable
parabolic defects.  Their orbits converge to the neutral fixed point
with polynomial or logarithmic rates.

The distinction is structural: the order of the first nonlinear
term determines the local parabolic scale, while the accumulation of
its non-autonomous coefficient determines whether the total drift is
finite or infinite.

% ---------------------------------------------------------------
\subsection{Relation with the autonomous return-map theory}
\label{subsec:relation-autonomous}
% ---------------------------------------------------------------

The transition maps studied here should be distinguished from the
autonomous return map associated with one fixed pair
\[
C\subset\Omega.
\]

In the autonomous theory, a critical point of the thickness function
is a fixed point of the return map and its derivative satisfies
\[
DF(c_*)=I+B_*H_*.
\]
The local spectrum is therefore governed by the Hessian of the
thickness and by the positive geometric operator \(B_*\).  Under the
appropriate nondegeneracy and strict-ascent assumptions, this produces
the stable--unstable splitting developed in the earlier papers.

The present transition
\[
\mathcal G_k:\partial C_k\longrightarrow\partial C_{k+1}
\]
has a different origin.  The two consecutive boundaries are tangent
at \(p\), and Section~\ref{sec:geometric-transition} gives directly
\[
G_k(s)-s=O(s^3).
\]
Hence
\[
G_k(0)=0,
\qquad
G_k'(0)=1,
\qquad
G_k''(0)=0.
\]
The neutral derivative is therefore a consequence of the tangency
geometry itself; it is not obtained by regarding \(G_k\) as an
autonomous return map with vanishing Hessian.

The first nontrivial dynamical information appears one order later,
through
\[
\beta_k
=
\frac{
	(\kappa_{k+1}-\kappa_k)
	(\kappa_{k+1}-2\kappa_k)
}{6}.
\]
Thus the threshold
\[
\kappa_{k+1}=2\kappa_k
\]
is a cubic-degeneracy threshold for the transition map.  It should
not be identified with the spectral threshold of the autonomous
theory, although both arise from geometric combinations controlling
a change of local dynamical behavior.

% ---------------------------------------------------------------
\subsection{Geometric realizability of critical regimes}
\label{subsec:critical-geometric-realizability}
% ---------------------------------------------------------------

The abstract logarithmic laws of Section~5 require more than the
formal behavior of one Taylor coefficient.

For example, the quintic critical regime assumes
\[
\gamma_k\sim-\frac{c}{k+1},
\]
together with a remainder satisfying a scale such as
\[
|R_k(s)|
\le
\frac{C}{k+1}|s|^7.
\]
Similarly, the quartic critical regime requires
\[
\delta_k\sim-\frac{c}{k+1}
\]
and weighted control of the fifth-order remainder.

The central realizability question is therefore:

\begin{openproblem}[Geometric realization of critical parabolic drift]
	\label{op:critical-realizability}
	Determine whether there exist nested convex boundaries with a common
	tangency point for which the induced transition maps satisfy either
	\[
	G_k(s)
	=
	s-\frac{c}{k+1}s^5
	+
	O\!\left(\frac1{k+1}|s|^7\right)
	\]
	in the symmetric case, or
	\[
	G_k(s)
	=
	s-\frac{c}{k+1}s^4
	+
	O\!\left(\frac1{k+1}|s|^5\right)
	\]
	in the nonsymmetric case, with \(c>0\).
\end{openproblem}

This problem is stronger than prescribing a sequence of coefficients.
The full transition map must arise from consecutive convex boundary
jets and must satisfy the associated nesting and first-intersection
geometry.

The diagonal-vanishing principle of
Section~\ref{subsec:higher-order-diagonal} suggests that critical
behavior should be studied through the rate at which consecutive
boundary jets approach one another.

% ---------------------------------------------------------------
\subsection{Beyond summable geometric variation}
\label{subsec:beyond-summable-variation}
% ---------------------------------------------------------------

The condition
\[
\sum_k\varepsilon_k<\infty,
\qquad
\varepsilon_k
=
\|f_{k+1}-f_k\|_{C^2},
\]
is sufficient for finite accumulated drift, but it is not expected
to be optimal.

The transition estimate gives
\[
|s_{k+1}-s_k|
\le
C\varepsilon_k |s_k|^3.
\]
Thus when
\[
\sum_k\varepsilon_k=\infty,
\]
the behavior depends on the interaction between the divergence of
the geometric variation and the shrinking of \(s_k\).

This suggests a broader classification problem.

\begin{openproblem}[Borderline geometric variation]
	\label{op:borderline-geometric-variation}
	Classify the local dynamics when
	\[
	\varepsilon_k\to0,
	\qquad
	\sum_k\varepsilon_k=\infty.
	\]
	In particular, determine geometric conditions under which the orbit
	converges to the common tangency point, converges to a nonzero limit,
	or exhibits a different critical regime.
\end{openproblem}

A particularly natural family is
\[
\varepsilon_k\asymp k^{-\alpha}.
\]
The abstract recurrence theory suggests transitions according to the
effective order of the first nonvanishing nonlinear term, but a
geometric theorem requires control of the complete transition defect,
not merely of individual coefficients.

% ---------------------------------------------------------------
\subsection{Nonsymmetric geometry}
\label{subsec:nonsymmetric-open}
% ---------------------------------------------------------------

Local reflection symmetry suppresses all even powers in the scalar
transition map.  Without symmetry one has
\[
G_k(s)
=
s+\beta_k s^3+\delta_k s^4+O(s^5),
\]
where
\[
\delta_k
=
\frac{
	3a_k(\kappa_{k+1}-\kappa_k)
	+
	(a_{k+1}-a_k)(4\kappa_k-3\kappa_{k+1})
}{24}.
\]

The quartic term can therefore become the first effective correction
when the cubic contribution is degenerate or sufficiently small.
Its sign depends not only on curvature variation but also on the
variation of the third boundary jet.

\begin{openproblem}[Geometric quartic regime]
	\label{op:geometric-quartic}
	Determine sharp geometric conditions on
	\[
	\kappa_{k+1}-\kappa_k
	\qquad\text{and}\qquad
	a_{k+1}-a_k
	\]
	under which the quartic term governs the asymptotic dynamics.
	
	In particular, determine whether the abstract logarithmic law
	\[
	s_k\sim(3c\log k)^{-1/3}
	\]
	can be realized by a nested nonsymmetric convex family.
\end{openproblem}

This problem also provides a natural test of how much information is
lost when only curvature is retained and higher boundary jets are
ignored.

% ---------------------------------------------------------------
\subsection{Higher-dimensional transition geometry}
\label{subsec:higher-dimensional-extension}
% ---------------------------------------------------------------

The explicit normal forms of Sections 2--3 were derived in dimension
two.  In higher dimension, let
\[
C_{k+1}\subset C_k\subset\mathbb R^N
\]
have a common tangency point and common tangent hyperplane.

The first-intersection construction and offset-independence remain
geometrically meaningful, but the scalar signed-arclength coordinate
must be replaced by local coordinates on the common tangent space
\[
T_p\partial C_k\simeq\mathbb R^{N-1}.
\]
The transition map then has the form
\[
G_k(u)
=
u+\mathcal B_k(u,u,u)+\text{higher-order terms},
\]
where the first nonidentity term, when cubic, is tensorial rather
than scalar.

\begin{openproblem}[Higher-dimensional cubic normal form]
	\label{op:higher-dimensional-normal-form}
	Derive an invariant expression for the cubic jet of the transition
	map in terms of the second fundamental forms, or shape operators, of
	\(\partial C_k\) and \(\partial C_{k+1}\) at the common tangency
	point.
	
	Determine its degeneracy set and the corresponding local
	non-autonomous dynamics.
\end{openproblem}

A satisfactory higher-dimensional theory should distinguish
direction-dependent attraction and repulsion, possible degeneracies
of the cubic tensor, and the role of noncommuting shape operators.

% ---------------------------------------------------------------
\subsection{Limit sets and the fractal question}
\label{subsec:fractal-question}
% ---------------------------------------------------------------

In the geometric regime proved in this paper, every sufficiently small
orbit converges to a single limit.  Hence its $\omega$-limit set is a
singleton and has Hausdorff dimension zero.  No nontrivial fractal limit set
is produced by the present theory.

A genuinely nontrivial $\omega$-limit set would require additional
dynamical mechanisms, such as branching, competing asymptotic channels,
repeated folding, or another source of recurrent non-autonomous behavior.

\begin{openproblem}[Nontrivial $\omega$-limit sets]
	Construct a geometrically realizable nested convex system for which the
	scalar non-autonomous transition orbit has a non-singleton $\omega$-limit
	set.
	
	If such examples exist, determine which geometric quantities control the
	topology and Hausdorff dimension of the resulting $\omega$-limit set.
\end{openproblem}

This problem is deliberately separated from the local parabolic
theory developed above.  The cubic tangency law alone does not imply
fractal behavior.

% ---------------------------------------------------------------
\subsection{The asymptotic displacement map}
\label{subsec:asymptotic-map-open}
% ---------------------------------------------------------------

Under summable geometric variation, Corollary~\ref{cor:limiting-map-cubic}
gives
\[
\mathcal L(s)
=
s+B s^3+o(s^3),
\qquad
B=\sum_{k=0}^{\infty}\beta_k.
\]
Thus the cubic contribution to the accumulated non-autonomous displacement
is exactly additive.

This raises a finer inverse question: to what extent do the higher-order
jets of \(\mathcal L\) encode the accumulated geometric variation of the
nested family? Beginning at fifth order, nonlinear composition generates
interaction terms between lower-order coefficients.

\begin{openproblem}[Higher jets of the limiting transition]
	Under stronger summability and regularity assumptions, determine the higher
	Taylor jet of the limiting map $\mathcal L$ in terms of the jets of the
	individual transitions $G_k$.
	
	The cubic contribution is additive by
	Corollary~\ref{cor:limiting-map-cubic}.  Beginning at fifth order, however,
	nonlinear composition generates interaction terms between lower-order
	coefficients.  Determine these interaction terms explicitly and identify
	geometric conditions under which the resulting higher-order coefficients
	encode accumulated variations of the boundary jets.
\end{openproblem}

This would connect the local coefficient theory of Section~3 with
the global accumulated dynamics of Section~4.
% ================================================================
\section{Conclusion}
\label{sec:conclusion}
% ================================================================

We have developed a local theory of non-autonomous return dynamics generated
by successive normal transitions between nested convex boundaries with a
common tangency point.

The first structural result is that the transition map between two
consecutive boundaries is independent of the auxiliary parallel
offset used in its construction.  More importantly, if
\[
\varepsilon_k
=
\|f_{k+1}-f_k\|_{C^2},
\]
then the complete scalar transition satisfies
\[
|G_k(s)-s|
\le
C\varepsilon_k |s|^3.
\]
Thus the parabolic character of the dynamics follows directly from
the geometry of consecutive tangent boundaries, and the entire
transition defect is controlled by their geometric variation.

Refining this estimate at the level of boundary jets yields the cubic
tangency law
\[
\beta_k
=
\frac{
	(\kappa_{k+1}-\kappa_k)
	(\kappa_{k+1}-2\kappa_k)
}{6}.
\]
Without local reflection symmetry, the next contribution is the
quartic term
\[
\delta_k
=
\frac{
	3a_k(\kappa_{k+1}-\kappa_k)
	+
	(a_{k+1}-a_k)(4\kappa_k-3\kappa_{k+1})
}{24}.
\]
Both formulas fit a general diagonal-vanishing principle: every
nonidentity Taylor coefficient of the transition map disappears when
the corresponding jets of two consecutive boundaries coincide.

The main geometric dynamical consequence is obtained under summable
variation,
\[
\sum_{k=0}^{\infty}
\|f_{k+1}-f_k\|_{C^2}
<\infty.
\]
In this regime the complete transition defect is summable.  Every
sufficiently small nonzero orbit has finite total variation and
converges to a nonzero limit,
\[
s_k\longrightarrow s_\infty\neq0,
\]
with
\[
|s_\infty-s_0|
=
O(|s_0|^3).
\]
Accordingly, the common tangency point is Lyapunov stable but not
locally attractive.  The accumulated non-autonomous dynamics defines
a limiting displacement map
\[
\mathcal L(s)
=
\lim_{n\to\infty}
G_{n-1}\circ\cdots\circ G_0(s),
\qquad
\mathcal L(s)
=
s+
\left(\sum_{k=0}^{\infty}\beta_k\right)s^3
+o(s^3).
\]

This summable geometric regime must be distinguished from critical
non-autonomous recurrences whose effective nonlinear drift is not
summable.  For an abstract recurrence of the form
\[
s_{k+1}-s_k
\sim
-a_k s_k^{m+1},
\]
the natural transformed variable satisfies
\[
s_k^{-m}
\sim
m\sum_{j<k}a_j.
\]
This yields, under remainder hypotheses matched to the scale of the
leading drift, the cubic, quartic, and quintic parabolic laws
established in Section~\ref{sec:abstract-recurrences}.  In
particular, harmonic coefficients produce the logarithmic rates
\[
s_k\sim(3c\log k)^{-1/3}
\]
in the quartic case and
\[
s_k\sim(4c\log k)^{-1/4}
\]
in the quintic case.

The explicit nested parabolic boundary-germ model illustrates the local
geometric regime without replacing it by a truncated recurrence: the exact
normal-intersection dynamics produces a finite displacement and
convergence to \(s_\infty\neq0\).  The logarithmic examples, by
contrast, are deliberately synthetic and illustrate only the
abstract critical mechanisms.

The remaining question is therefore not the derivation of further
formal asymptotic rates, but their geometric realizability.  In
particular, it remains to determine which nonsummable variations of
nested convex boundaries can generate critical quartic or quintic
drift with the required control of the complete transition remainder.
Higher-dimensional transition geometry and genuinely nontrivial
limit sets provide further directions beyond the present
two-dimensional parabolic theory.
% ============================================================

\end{document}